\documentclass{article}
\usepackage[utf8]{inputenc}

\usepackage{amsmath}
\usepackage{amsfonts}
\usepackage{amssymb}
\usepackage{amsthm}
\usepackage{amsrefs}
\usepackage{tikz}
\usepackage{tikz-cd}
\usepackage{pdfpages}
\usepackage{enumerate}
\usepackage{geometry}
\usepackage{rotating}
\usepackage{xcolor,colortbl}
\usepackage{tabularx}

\usetikzlibrary{matrix,arrows,calc,shapes.geometric,decorations.pathmorphing}

\DeclareMathOperator{\lcm}{lcm}
\DeclareMathOperator{\tor}{Tor}
\let \gcd \relax
\DeclareMathOperator{\gcd}{gcd}
\DeclareMathOperator{\cl}{cl}
\DeclareMathOperator{\m}{\mathcal{M}}
\DeclareMathOperator{\N}{\mathbb{N}}
\DeclareMathOperator{\Z}{\mathbb{Z}}
\let \dim \relax
\DeclareMathOperator{\dim}{dim}

\DeclareMathOperator{\sgn}{sgn}
\let \ker \relax
\DeclareMathOperator{\ker}{ker}
\DeclareMathOperator{\im}{im}

\DeclareMathOperator{\colim}{colim}
\DeclareMathOperator{\supp}{supp}
\DeclareMathOperator{\mdeg}{mdeg}

\theoremstyle{definition}
\newtheorem{definition}{Definition}[section]
\newtheorem{theorem}[definition]{Theorem}
\newtheorem{lemma}[definition]{Lemma}
\newtheorem{corollary}[definition]{Corollary}

\newtheorem{example}[definition]{Example}
\newtheorem{remark}[definition]{Remark}
\newtheorem{construction}[definition]{Construction}

\newtheorem{maintheorem}{Theorem}

\begin{document}

\title{Massey products and the Golod property for simplicially resolvable rings}
\author{Robin Frankhuizen}

\maketitle

\begin{abstract}
We apply algebraic Morse theory to the Taylor resolution of a monomial ring $R=S/I$ to obtain an $A_{\infty}$-structure on the minimal free resolution of $R$. Using this structure we describe the vanishing of higher Massey products in case the minimal free resolution is simplicial. Under this assumption, we show that $R$ is Golod if and only if the product on $\tor^S(R,k)$ vanishes. Lastly, we give two combinatorial characterizations of the Golod property in this case. 
\end{abstract}

\section{Introduction}
Let $I = (m_1,\ldots, m_r)$ be an ideal generated by monomials in the polynomial algebra $S=k[x_1,\ldots,x_m]$ over a field $k$. The quotient $R = S/I$ is called a \emph{monomial} ring. Define formal power series by 
$$ P^R_k(t) = \sum_{j=0}^{\infty} \dim \tor^R_j(k,k) t^j \quad \mbox{ and } \quad P^S_R(t) = \sum_{j=0}^{\infty}\dim \tor^S_j(R,k) t^j.$$
The first of these is called the \emph{Poincar\'e series} of $R$. A result due to Serre \cite{serre1965} states that there is an inequality of power series
\begin{equation}
\label{serreinequality}
 P^R_k(t) \leq \frac{(1+t)^m}{1-t(P^S_R(t) -1)}.
\end{equation}
The problem of when equality is obtained goes back to at least the 70s when Golod \cite{golod1978} showed that \eqref{serreinequality} is an equality if and only if all Massey products on the Tor-algebra $\tor^S(R,k)$ vanish. In honor of this result, a monomial ring $R$ is called \emph{Golod} if \eqref{serreinequality} is an equality. In general, it is hard to directly verify the vanishing of Massey products and so in practice the Golod property is still hard to determine. 

Though the Golod property has been studied in commutative algebra from the 70s, it has recently received an increasing amount of attention in topology. The Tor-algebra shows up here in the cohomology of the so-called moment-angle complexes which are defined as follows. Let $\Delta$ be a simplicial complex on vertex set $[m] = \{1, \ldots, m \}$ and define the \emph{moment-angle complex} $Z_{\Delta}$ as follows. Let $D^2$ denote the $2$-disc and $S^1$ its bounding circle. For $\sigma \in \Delta$, define
$$X_{\sigma} = \prod_{i=1}^m Y_i \subseteq (D^2)^m \quad \mbox{ where } \quad Y_i = \begin{cases} D^2 &\mbox{ if } i \in \sigma \\ S^1 &\mbox{ if } i \notin \sigma \end{cases}$$
Lastly, we put
$$Z_{\Delta} = \colim_{\sigma \in \Delta} X_{\sigma} \subseteq (D^2)^m.$$
On the other hand, given a simplicial complex $\Delta$, the \emph{Stanley-Reisner ring} $k[\Delta]$ is defined as
$$k[\Delta] = S / (x_{i_1}\cdots x_{i_k} \mid \{i_1,\ldots,i_k \} \notin \Delta ).$$
Note that $k[\Delta]$ is a square-free monomial ring. The moment-angle complex $Z_{\Delta}$ and the Stanley-Reisner ring $k[\Delta]$ are related by a result of Baskakov, Buchstaber and Panov \cite{baskakovbuchstaberpanov2004} which states that there is an isomorphism of graded algebras
$$H^*(Z_{\Delta},k) \cong \tor^S(k[\Delta],k).$$
The homotopy type of $Z_{\Delta}$ is as of yet not well understood, but significant progress has been made for those $Z_{\Delta}$ where $\Delta$ is Golod, see for example \cite{grbictheriault2007}, \cite{grbictheriault2013}, \cite{iriyekishimoto2013b} and \cite{bebengrbic2017}.

The preceding discussion makes clear that the Golod property is of interest in both commutative algebra and algebraic topology and as a consequence at lot of attention has been devoted to find characterizations of Golodness. For example, a combinatorial characterization of Golodness in terms of the homology of the lower intervals in the lattice of saturated subsets is given in \cite{berglund2006}. It has been claimed in \cite{berglundjollenbeck2007} that $R$ is Golod if and only if the product on $\tor^S(R,k)$ vanishes. However, recently a counterexample to this claim was found in \cite{katthan2015} where the error is traced back to \cite{jollenbeck2006}. 

In \cite{frankhuizen2016}, we developed an approach to study Massey products on $\tor^S(R,k)$ using $A_{\infty}$-algebras and applied this to give necessary and sufficient conditions for Golodness for rooted rings. The purpose of this paper is to extend the methods developed in \cite{frankhuizen2016} to monomial rings whose minimal free resolution is simplicial in the sense of \cite{bayerpeevasturmfels1998}. 

The main idea is the following. As a consequence of the result in \cite{frankhuizen2016}, we can study the Golod property in terms of $A_{\infty}$-structures on the minimal free resolution of $R$. In this paper, we construct such $A_{\infty}$-structures by applying algebraic Morse theory to the Taylor resolution of a monomial ring $R$. When $R$ is simplicially resolvable (see Definition \ref{definitionsimpliciallyresolvable}), it turns out that this structure is a comparatively simple description which is given in Lemma \ref{ainfinitystructuremorsecomplexsimpliciallyresolvable}. By using this description, we obtain the first main result of this paper. 

\begin{maintheorem}
\label{maintheorem1}
Let $R=S/I$ be simplicially resolvable. Then the following are equivalent.
\begin{enumerate}
\item $R$ is Golod
\item The product on $\tor^S(R,k)$ is trivial.
\item $I$ satisfies the gcd condition. That is, for any two generators $m_1$ and $m_2$ of $I$ with $\gcd(m_1,m_2)=1$ there exists a generator $m \neq m_1,m_2$ such that $m$ divides $\lcm(m_1,m_2)$.
\item For $u,v \in \m_0$ we have $\lcm(u)\lcm(v) \neq \lcm(uv)$ whenever $uv \in \m_0$. 
\end{enumerate}
\end{maintheorem}

In particular, the main result from \cite{berglundjollenbeck2007} does hold when restricted to simplicially resolvable rings. Next, we turn our attention to the vanishing of higher Massey products. We show that a sufficient condition for the vanishing of higher Massey products is the existence of a standard Morse matchings which were first introduced in \cite{jollenbeck2006}. More precisely, we prove the following result. 

\begin{maintheorem}
\label{maintheorem2}
Let $R$ be simplicially resolvable. Suppose that the Taylor resolution $T$ admits a standard matching. Then all higher Massey products are trivial.
\end{maintheorem}

\section{Simplicial Resolutions}

The following method of constructing free resolutions of monomial rings is due to Bayer, Peeva and Sturmfels \cite{bayerpeevasturmfels1998}. A leisurely introduction can be found in \cite{mermin2012}. Let $\lbrace m_1, \ldots, m_r \rbrace$ be a set of monomials. Fix some total order $\prec$ on $\lbrace m_1, \ldots, m_r \rbrace$. 
After relabelling we may assume that $m_1 \prec m_2 \prec \cdots \prec m_r$. Let $\Delta$ be a simplicial complex on the vertex set $\lbrace 1,\ldots,r \rbrace$. 
By abuse of notation, we will say $\Delta$ is a simplicial complex on vertex set $\lbrace m_1, \ldots, m_r \rbrace$. Assign a multidegree $m_J$ to each simplex $J \in \Delta$ by defining 
$$m_J = \lcm \lbrace m_j \mid j \in J \rbrace.$$
Define a chain complex $F_{\Delta}$ associated to $\Delta$ as follows. 
Let $F_n$ be the free $S$-module on generators $u_J$ with $\vert J \vert = n$.
For $J = \lbrace j_1 \prec \cdots \prec j_n \rbrace$, put $J^i = \lbrace j_1 \prec \cdots \prec \widehat{j_i} \prec \cdots \prec j_n \rbrace$. 
The differential $d\colon F_n \to F_{n-1}$ is defined, for $J \in \Delta$, by 
$$ d(u_J) = \sum_{i=1}^{\vert J \vert} (-1)^{i+1} \frac{m_J}{m_{J^i}} u_{J^i}.$$

In general, $F_{\Delta}$ need not be a resolution of $S / I$. However, we have the following theorem.

\begin{theorem}[\cite{bayerpeevasturmfels1998}, Lemma 2.2]
\label{bayerpeevasturmfelstheorem}
 Let $\Delta$ be a simplicial complex on vertex set $\lbrace m_1, \ldots, m_r \rbrace$ and define, for a multidegree $\mu$, a subcomplex 
$$ \Delta_{\mu} = \lbrace J \in \Delta \mid m_J \mbox{ divides } \mu \rbrace.$$
 Then $F_{\Delta}$ is a resolution of $R$ if and only if $\Delta_{\mu}$ is either acyclic or empty for all multidegrees $\mu$. 
\end{theorem}

\begin{definition}
\label{definitionsimplicialresolution}
A resolution $F$ is called a \emph{simplicial resolution} if $F = F_{\Delta}$ for some simplicial complex $\Delta$.
\end{definition}

An important special case of the above construction is the following. Let $\Delta = \Delta^r$ be the full $r$-simplex. Then $\Delta$ satisfies the assumptions of Theorem \ref{bayerpeevasturmfelstheorem}. Therefore, $F_{\Delta}$ is a simplicial resolution which is called the \emph{Taylor resolution} \cite{taylor1966} and will be denoted by $T$. The Taylor resolution admits a multiplication defined by
$$ u_I \cdot u_J = \begin{cases} \sgn(I,J) \frac{m_Im_J}{m_{I\cup J}} u_{I \cup J} &\mbox{ if } I \cap J = \emptyset \\ 0 &\mbox{ otherwise} \end{cases}$$
where $\sgn(I,J)$ is the sign of the permutation making $I \cup J$ into an increasing sequence.
This multiplication induces a differential graded algebra (dga) structure on $T$. The \emph{Tor-algebra} $\tor^S(S/I,k)$ of $S/I$ is 
$$ \tor^S(S/I,k) = \bigoplus_{n} \tor^S_n(S/I,k) = \bigoplus_n H_n(T \otimes_S k)$$
where the multiplication is induced by the multiplication on $T$\\


In this paper we will only consider resolutions $F$ that are as small as possible in the sense that each $F_n$ has the minimal number of generators. 
More precisely, we have the following definition. 

\begin{definition}
Let $S/I$ be a monomial ring. A free resolution $F \to S/I$ is said to be \emph{minimal} if $d(F) \subseteq (x_1,\ldots,x_m)F$. 
\end{definition}

\section{Massey products and $A_{\infty}$-algebras}
We briefly recall Massey products which were first introduced in \cite{massey1958}.

\begin{definition}
Let $(A,d)$ be a differential graded algebra. If $a \in A$, we write $\bar{a}$ for $(-1)^{\text{deg}(a) +1}a$. \\
Let $\alpha_1,\alpha_2 \in HA$. The length $2$ \textit{Massey product} $\langle \alpha_1, \alpha_2 \rangle$ is defined to be the product $\alpha_1 \alpha_2$ in the homology algebra $HA$. \\
Let $\alpha_1, \ldots, \alpha_n \in HA$ be homology classes with the property that each length $j-i+1$ Massey product $\langle \alpha_i, \ldots, \alpha_j \rangle$ is defined and contains zero for $i<j$ and $j-i < n-1$. A \emph{defining system} $\{ a_{ij} \}$ consists of
\begin{enumerate}
\item For $i=1,\ldots,n$, representing cycles $a_{i-1,i}$ of the homology class $\alpha_i$.
\item For $j > i+1$, elements $a_{ij}$ such that 
$$da_{ij} = \sum_{i<k<j} \bar{a}_{ik}a_{kj}.$$ 
\end{enumerate}
Note that the existence is guaranteed by the condition that $\langle \alpha_i, \ldots, \alpha_j \rangle$ is defined and contains zero for $i<j$ and $j-i < n-1$.
The length $n$ \textit{Massey product} $\langle \alpha_1, \ldots, \alpha_n \rangle$  is defined as the set
$$ \langle \alpha_1, \ldots, \alpha_n \rangle  = \{ [\sum_{0<i<n} \bar{a}_{0i} a_{in}]\mid \{ a_{ij} \} \mbox{ is a defining system } \} \subseteq H^{s+2-n}$$
where $s = \sum_{i=1}^n \deg \alpha_i$. 
\end{definition}
A Massey product $\langle \alpha_1, \ldots, \alpha_n \rangle$ is said to be \emph{trivial} if it contains zero. 

\begin{theorem}[\cite{golod1978}, see also Section 4.2 of \cite{gulliksenlevin1969}]
\label{golodiffmasseytrivial}
Let $R$ be a monomial ring. Then $R$ is Golod if and only if all Massey products on the Koszul homology $\tor^S(R,k)$ are trivial.
\end{theorem}

Following \cite{katthan2016}, we will say that a dga $A$ satisfies \emph{condition} $(B_r)$ if all $k$-ary Massey products are defined and contain only zero for all $k \leq r$. Let $R$ be a monomial ring and let $K_S$ be the Koszul resolution of the base field $k$ over $S$. The \emph{Koszul dga} $K_R$ of $R$ is defined as $K_R = R \otimes_S K_S$. Again following \cite{katthan2016}, we say that a monomial ring $R$ satisfies $(B_r)$ if the dga $K_R$ of $R$ satisfies $(B_r)$.

\begin{lemma}
 Let $R$ be a monomial ring. Then $R$ is Golod if and only if $R$ satisfies condition $(B_r)$ for all $r \in \N$. 
\end{lemma}

In general, it is very hard to study Massey products directly. In the remainder of this section we will discuss how $A_{\infty}$-algebras provide an alternative way of studying Massey products. Since their introduction by Stasheff \cite{stasheff1963}, $A_{\infty}$-algebras have found applications in various branches of mathematics. A general overview can be found in \cite{keller2001}. Our exposition follows that of \cite{lupalmieriwuzhang2009}. All signs will be determined by the \emph{Koszul sign convention}
\begin{equation}
\label{koszulsignconvention}
(f \otimes g) (x \otimes y) = (-1)^{\vert g \vert \cdot \vert x \vert} fx \otimes gy. 
\end{equation}

\begin{definition}
Let $R$ be a commutative ring and $A = \oplus A_n$ a $\Z$-graded free $R$-module. An $A_{\infty}$-algebra structure on $A$ consists of maps $\mu_n\colon A^{\otimes n} \to A$ for each $n \geq 1$ of degree $n-2$ satisfying the \emph{Stasheff identities}
\begin{equation}
\label{stasheffidentities}
 \sum (-1)^{r+st} \mu_u(1^{\otimes r} \otimes \mu_s \otimes 1^{\otimes t}) = 0
\end{equation}
where the sum runs over all decompositions $n=r+s+t$ with $r,t \geq 0$, $s \geq 1$ and $u=r+t+1$. 
\end{definition}
Observe that when applying (\ref{stasheffidentities}) to an element additional signs appear because of the Koszul sign convention (\ref{koszulsignconvention}). In the special case when $\mu_3=0$, it follows that $\mu_2$ is strictly associative and so $A$ is a differential graded algebra with differential $\mu_1$ and multiplication $\mu_2$. An $A_{\infty}$-algebra $A$ is called \emph{strictly unital} if there exists an element $1 \in A$ that is a unit for $\mu_2$ and such that for all $n \neq 2$
$$\mu_n(a_1 \otimes \cdots \otimes a_n) = 0$$
whenever $a_i=1$ for some $i$. 



The notion of a morphism between $A_{\infty}$-algebras will also be needed. 
\begin{definition}
 Let $(A, \mu_n)$ and $(B,\overline{\mu}_n)$ be $A_{\infty}$-algebras. A \emph{morphism} of $A_{\infty}$-algebras (or $A_{\infty}$\emph{-morphism}) $f\colon A \to B$ is a family of linear maps 
 \[
  f_n\colon A^{\otimes n} \to B
 \]
of degree $n-1$ satisfying the \emph{Stasheff morphism identities}
\begin{equation}
 \label{stasheffmorphismidentities}
 \sum (-1)^{r+st}f_u(1^{\otimes r} \otimes \mu_s \otimes 1^{\otimes t}) = \sum (-1)^w \overline{\mu}_q(f_{i_1} \otimes f_{i_2} \otimes \cdots \otimes f_{i_q})
\end{equation}
for every $n \geq 1$. The first sum runs over all decompositions $n=r+s+t$ with $s \geq 1$ and $r,t \geq 0$ where $u=r+t+1$. The second sum runs over all $1 \leq q \leq n$ and all decompositions $n = i_1 + i_2 + \cdots + i_q$ with all $i_s \geq 1$. 
The sign on the right-hand side of (\ref{stasheffmorphismidentities}) is given by
\[
 w = \sum_{p=1}^{q-1}(q-p)(i_p-1).
\]
If $A$ and $B$ are strictly unital, an $A_{\infty}$-morphism is also required to satisfy $f_1(1) = 1$ and
\[
 f_n(a_1 \otimes \cdots \otimes a_n) = 0
\]
if $n \geq 2$ and $a_i = 1$ for some $i$. 
\end{definition}
A morphism $f$ is called a \emph{quasi-isomorphism} if $f_1$ is a quasi-isomorphism in the usual sense, that is if $f_1$ induces an isomorphism in homology.
Let $A$ be an $A_{\infty}$-algebra. Then its homology $HA$ is an associative algebra. A crucial result relating the $A_{\infty}$-algebra $A$ and its homology algebra $HA$ is the \emph{homotopy transfer theorem}. 
\begin{theorem}[Homotopy Transfer Theorem, \cite{kadeishvili1980}, see also \cite{merkulov1999}] 
\label{homotopytransfertheorem}
Let $(A, \mu_n)$ be an $A_{\infty}$-algebra over a field $R$ and let $HA$ be its homology algebra. 
There exists an $A_{\infty}$-algebra structure $\mu'_n$ on $HA$ such that
\begin{enumerate}
 \item $\mu'_1 = 0$, $\mu'_2 = H(\mu_2)$ and the higher $\mu'_n$ are determined by $\mu_n$
 \item there exists an $A_{\infty}$-quasi-isomorphism $HA \to A$ lifting the identity morphism of $HA$.
\end{enumerate}
Moreover, this $A_{\infty}$-structure is unique up to isomorphism of $A_{\infty}$-algebras.
\end{theorem}

If $A$ is a dga then a more explicit way of constructing $A_{\infty}$-structures on homology is available which is originally due to Merkulov \cite{merkulov1999}. We first need the following definition.

\begin{definition}
Let $A$ be a chain complex and $B \subseteq A$ a subcomplex. A \emph{transfer diagram} is a diagram of the form
\begin{equation}
\label{transferdiagram}
\begin{tikzpicture}[baseline=(current  bounding  box.center)]
\matrix(m)[matrix of math nodes,
row sep=3em, column sep=2.8em,
text height=1.5ex, text depth=0.25ex]
{B &A\\};
\path[->]
(m-1-1) edge [bend left=35] node[yshift=1.5ex] {$i$} (m-1-2)
(m-1-2) edge [bend left=35] node[yshift=-1.5ex] {$p$} (m-1-1)
(m-1-2) edge [loop right, in=35,out=-35,looseness=5, min distance=10mm] node {$\phi$} (m-1-2)
;
\end{tikzpicture}
\end{equation}
where $pi = 1_B$ and $ip - 1 = d \phi + \phi d$. 
\end{definition}

Some authors use the term strong deformation retract for what we call a transfer diagram. 

\begin{theorem}[\cite{merkulov1999}, Theorem 3.4]
\label{merkulovtheorem}
Let $(A,d)$ be a dga and $B$ a subcomplex of $A$ such that there exists a transfer diagram of the form \eqref{transferdiagram}. Define linear maps $\lambda_n\colon A^{\otimes n} \to A$ as follows. We let $\lambda_2$ denote the product in $A$ and we set
\begin{equation}
\label{merkulovlambda}
\lambda_n = \sum_{\substack{s+t = n \\ s,t \geq 1}} (-1)^{s+1} \lambda_2 (\phi \lambda_s, \phi \lambda_t)
\end{equation}
Now, define a second series of maps $\mu_n\colon B^{\otimes n} \to B$ by setting $\mu_1 = d$ and, for $n \geq 2$, 
\begin{equation}
\label{merkulovmun}
\mu_n = p \circ \lambda_n \circ i^{\otimes n}.
\end{equation}
Then $(B,\mu_n)$ is an $A_{\infty}$-algebra.  
\end{theorem}

Now, let $R=S/I$ be a monomial ring. We will say that a map of $S$-modules $f\colon M \to N$ is \emph{minimal} if $f \otimes 1 \colon M \otimes_S k \to N \otimes_S k$ is zero. It is readily verified that $f$ is minimal if and only if $f$ maps into $(x_1, \ldots, x_m)N$. The following theorem relates $A_{\infty}$-algebras and Massey products.

\begin{theorem}[\cite{frankhuizen2016}, Theorem 4.6]
 Let $R = S/I$ be a monomial ring with minimal free resolution $F$. Let $r \in \N$ and let $\mu_n$ be an $A_{\infty}$-structure on $F$ such that $F \otimes_S k$ and $K_R$ are quasi-isomorphic as $A_{\infty}$-algebras. Then $R$ satisfies $(B_r)$ if and only if $\mu_k$ is minimal for all $k \leq r$. 
\end{theorem}

In particular, the following corollary will be used extensively. 

\begin{corollary}[\cite{frankhuizen2016}, Corollary 4.6, see also \cite{burke2015}]
\label{munminimalimpliesgolod}
Let $R = S/I$ be a monomial ring with minimal free resolution $F$. Let $\mu_n$ be an $A_{\infty}$-structure on $F$ such that $F \otimes_S k$ and $K_R$ are quasi-isomorphic as $A_{\infty}$-algebras. Then $R$ is Golod if and only if $\mu_n$ is minimal for all $n \geq 1$.   
\end{corollary}
\section{Algebraic Morse theory}
In this section we recall algebraic Morse theory that was independently developed by Sk\"oldberg \cite{skoldberg2006} and J\"ollenbeck and Welker \cite{jollenbeckwelker2009} based on earlier work by Forman \cite{forman1995,forman2002}. Our exposition follows that of \cite{skoldberg2006}.

Let $R$ be a ring with unit. A \emph{based complex} $K$ is a chain complex $(K,d)$ together with a direct sum decomposition
$$ K_n = \bigoplus_{\alpha \in I_n} K_{\alpha}$$
where the $I_n$ are pairwise disjoint. We will write $\alpha^{(n)}$ to indicate that $\alpha \in I_n$.
Let $f\colon K \to K$ be a graded map. We write $f_{\beta, \alpha}$ for the component of $f$ going from $K_{\alpha}$ to $K_{\beta}$, that is $f_{\beta, \alpha}$ is the composition
\begin{center}
\begin{tikzpicture}[descr/.style={fill=white,inner sep=1.5pt}]
        \matrix (m) [
            matrix of math nodes,
            row sep=3em,
            column sep=2em,
            text height=1.5ex, text depth=0.25ex
        ]
        {K_{\alpha} & K_m & K_n & K_{\beta} \\
         };

        \path[overlay,->, font=\small, >=latex]
                        (m-1-1) edge (m-1-2) 
                        (m-1-2) edge node[yshift=1.5ex] {$f$} (m-1-3)
                        (m-1-3) edge (m-1-4)
                        ;
\end{tikzpicture}   
\end{center} 
where $K_{\alpha} \to K_m$ is the inclusion and $K_n \to K_{\beta}$ the projection. Given a based complex $K$, define a directed graph $G_K$ with vertex set $\cup_n I_n$ and a directed edge $\alpha \to \beta$ if $d_{\beta,\alpha} \neq 0$.
We will only consider situations in which $G_K$ is finite.\\

A \emph{partial matching} on a directed graph $D=(V,E)$ is a subset $\m$ of the edges $E$ such that no vertex is incident to more than one edge of $\m$. We define a new directed graph $D^{\m} = (V,E^{\m})$ by setting
$$ E^{\m} = (E \setminus \m) \cup \lbrace \beta \to \alpha \mid \alpha \to \beta \in \m \rbrace.$$
That is to say, $D^{\m}$ is the directed graph obtained from $D$ by inverting all the edges in $\m$.

\begin{definition}
 A partial matching $\m$ on a directed graph $G_K$ is a \emph{Morse matching} if
 \begin{enumerate}
 \item for each edge $\alpha \to \beta$ in $\m$, the component $d_{\beta,\alpha}$ is an isomorphism, 
 and
 \item $G_K^{\m}$ has no directed cycles.
 \end{enumerate}
\end{definition}


A vertex in $G_K^{\m}$ that is not matched by $\m$ is called $\m$-\emph{critical} and we write $\m^0$ for the set of $\m$-critical vertices. Define
\begin{center}
\begin{tikzpicture}[descr/.style={fill=white,inner sep=1.5pt}]
        \matrix (m) [
            matrix of math nodes,
            row sep=3em,
            column sep=2em,
            text height=1.5ex, text depth=0.25ex
        ]
        {\m^-=\{ \alpha \mid \beta \to \alpha \in \m \mbox{ for some } \beta \} &  \m^+=\{ \alpha \mid \alpha \to \gamma \in \m \mbox{ for some } \gamma \}\\
         };
\end{tikzpicture}   
\end{center}
We will also write
\begin{center}
\begin{tikzpicture}[descr/.style={fill=white,inner sep=1.5pt}]
        \matrix (m) [
            matrix of math nodes,
            row sep=3em,
            column sep=2em,
            text height=1.5ex, text depth=0.25ex
        ]
        {\m_n^0 = \m^0 \cap I_n & \m_n^- = \m^- \cap I_n & \m_n^+ = \m^+ \cap I_n\\
         };
\end{tikzpicture}   
\end{center}



\begin{definition}
Let $K$ be a based complex. Denote the edges of $G_K$ by $E$. A Morse matching $\m$ on $G_K$ is called \emph{maximal} if no proper super set $\m \subseteq \m' \subseteq E$ is a Morse matching. 
\end{definition}

\begin{example}
\label{5gonmorsetheory}
Let $\Delta$ be the $5$-gon labeled as
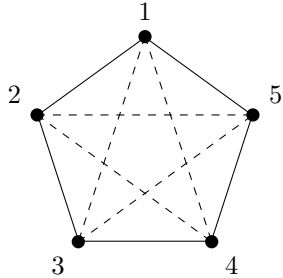
\begin{center}
\begin{tikzpicture}
\node[regular polygon, regular polygon sides=5, minimum size=3cm, draw] at (0,0) (A) {};
\foreach \i in {1,...,5}
    \node[circle, inner sep=0pt, minimum size=5pt, label=90 + (\i-1)*72:\i, fill=black] at (A.corner \i) {};
    \draw[dashed] (A.corner 1) -- (A.corner 3);
    \draw[dashed] (A.corner 1) -- (A.corner 4);
    \draw[dashed] (A.corner 2) -- (A.corner 4);
    \draw[dashed] (A.corner 2) -- (A.corner 5);
    \draw[dashed] (A.corner 3) -- (A.corner 5);
\end{tikzpicture}
\end{center}
The Stanley-Reisner ring of $\Delta$ is 
$$k[\Delta] = k[x_1,x_2,x_3,x_4,x_5] / (x_1x_3,x_1x_4,x_2x_4,x_2x_5,x_3x_5).$$
Let $T \to k[\Delta]$ denote the Taylor resolution of $k[\Delta]$. Denote the generators respectively by $u_1,\ldots,u_5$. Given $J = \{j_1 < \cdots j_k \}$, write $u_J = u_{j_1}\cdots u_{j_k}$. Figure \ref{5gonmorsegraph} depicts the graph $G_T$ corresponding to $I$. Here, the red arrows (both solid and dashed) are invertible. The solid red arrows give an example of a maximal Morse matching on $T$. 
\end{example}

\begin{figure}
\begin{center}
\begin{sideways}
\begin{tikzpicture}[descr/.style={fill=white,inner sep=1.5pt}]

\node (45) at (0,0) {$u_{45}$};
\node (35) at (2,0) {$u_{35}$};
\node (34) at (4,0) {$u_{34}$};
\node (25) at (6,0) {$u_{25}$};
\node (24) at (8,0) {$u_{24}$};
\node (23) at (10,0) {$u_{23}$};
\node (15) at (12,0) {$u_{15}$};
\node (14) at (14,0) {$u_{14}$};
\node (13) at (16,0) {$u_{13}$};
\node (12) at (18,0) {$u_{12}$};

\node (345) at (0,2) {$u_{345}$};
\node (245) at (2,2) {$u_{245}$};
\node (235) at (4,2) {$u_{235}$};
\node (234) at (6,2) {$u_{234}$};
\node (145) at (8,2) {$u_{145}$};
\node (135) at (10,2) {$u_{135}$};
\node (134) at (12,2) {$u_{134}$};
\node (125) at (14,2) {$u_{125}$};
\node (124) at (16,2) {$u_{124}$};
\node (123) at (18,2) {$u_{123}$};

\node (2345) at (2,4) {$u_{2345}$};
\node (1345) at (6,4) {$u_{1345}$};
\node (1245) at (9,4) {$u_{1245}$};
\node (1235) at (12,4) {$u_{1235}$};
\node (1234) at (16,4) {$u_{1234}$};

\node (5) at (2,-2) {$u_{5}$};
\node (4) at (6,-2) {$u_{4}$};
\node (3) at (9,-2) {$u_{3}$};
\node (2) at (12,-2) {$u_{2}$};
\node (1) at (16,-2) {$u_{1}$};

\node (12345) at (9,6) {$u_{12345}$};

\node (0) at (9,-4) {$1$};

\path [->] (5) edge[thick] (0);
\path [->] (4) edge[thick] (0);
\path [->] (3) edge[thick] (0);
\path [->] (2) edge[thick] (0);
\path [->] (1) edge[thick] (0);

\path [->] (45) edge[thick] (5);
\path [->] (45) edge[thick] (4);

\path [->] (35) edge[thick] (5);
\path [->] (35) edge[thick] (3);

\path [->] (34) edge[thick] (4);
\path [->] (34) edge[thick] (3);

\path [->] (25) edge[thick] (5);
\path [->] (25) edge[thick] (2);

\path [->] (24) edge[thick] (4);
\path [->] (24) edge[thick] (2);

\path [->] (23) edge[thick] (3);
\path [->] (23) edge[thick] (2);

\path [->] (15) edge[thick] (5);
\path [->] (15) edge[thick] (1);

\path [->] (14) edge[thick] (4);
\path [->] (14) edge[thick] (1);

\path [->] (13) edge[thick] (3);
\path [->] (13) edge[thick] (1);

\path [->] (12) edge[thick] (2);
\path [->] (12) edge[thick] (1);

\path [->] (345) edge[thick] (45);
\path [->] (345) edge[red,thick] (35);
\path [->] (345) edge[thick] (34);

\path [->] (245) edge[thick] (45);
\path [->] (245) edge[thick] (25);
\path [->] (245) edge[thick] (24);

\path [->] (235) edge[thick] (35);
\path [->] (235) edge[thick] (25);
\path [->] (235) edge[thick] (23);

\path [->] (234) edge[thick] (34);
\path [->] (234) edge[red,thick] (24);
\path [->] (234) edge[thick] (23);

\path [->] (145) edge[thick] (45);
\path [->] (145) edge[thick] (15);
\path [->] (145) edge[thick] (14);

\path [->] (135) edge[thick] (35);
\path [->] (135) edge[thick] (15);
\path [->] (135) edge[thick] (13);

\path [->] (134) edge[thick] (34);
\path [->] (134) edge[thick] (14);
\path [->] (134) edge[thick] (13);

\path [->] (125) edge[red,thick] (25);
\path [->] (125) edge[thick] (15);
\path [->] (125) edge[thick] (12);

\path [->] (124) edge[thick] (24);
\path [->] (124) edge[thick] (14);
\path [->] (124) edge[thick] (12);

\path [->] (123) edge[thick] (23);
\path [->] (123) edge[red,thick] (13);
\path [->] (123) edge[thick] (12);

\path [->] (2345) edge[thick] (345);
\path [->] (2345) edge[red,thick] (245);
\path [->] (2345) edge[red,dashed,thick] (235);
\path [->] (2345) edge[thick] (234);

\path [->] (1345) edge[thick] (345);
\path [->] (1345) edge[thick] (145);
\path [->] (1345) edge[red,thick] (135);
\path [->] (1345) edge[red,dashed,thick] (134);

\path [->] (1245) edge[red,dashed,thick] (245);
\path [->] (1245) edge[thick] (145);
\path [->] (1245) edge[thick] (125);
\path [->] (1245) edge[red,dashed,thick] (124);

\path [->] (1235) edge[red,dashed,thick] (235);
\path [->] (1235) edge[red,dashed,thick] (135);
\path [->] (1235) edge[thick] (125);
\path [->] (1235) edge[thick] (123);

\path [->] (1234) edge[thick] (234);
\path [->] (1234) edge[red,dashed,thick] (134);
\path [->] (1234) edge[red,thick] (124);
\path [->] (1234) edge[thick] (123);

\path [->] (12345) edge[red,dashed,thick] (2345);
\path [->] (12345) edge[red,dashed,thick] (1345);
\path [->] (12345) edge[red,dashed,thick] (1245);
\path [->] (12345) edge[red,thick] (1235);
\path [->] (12345) edge[red,dashed,thick] (1234);

\end{tikzpicture}   
\end{sideways}
\end{center} 
\caption{The graph $G_T$ corresponding to the face ideal $I$ of the $5$-gon.}
\label{5gonmorsegraph}
\end{figure}
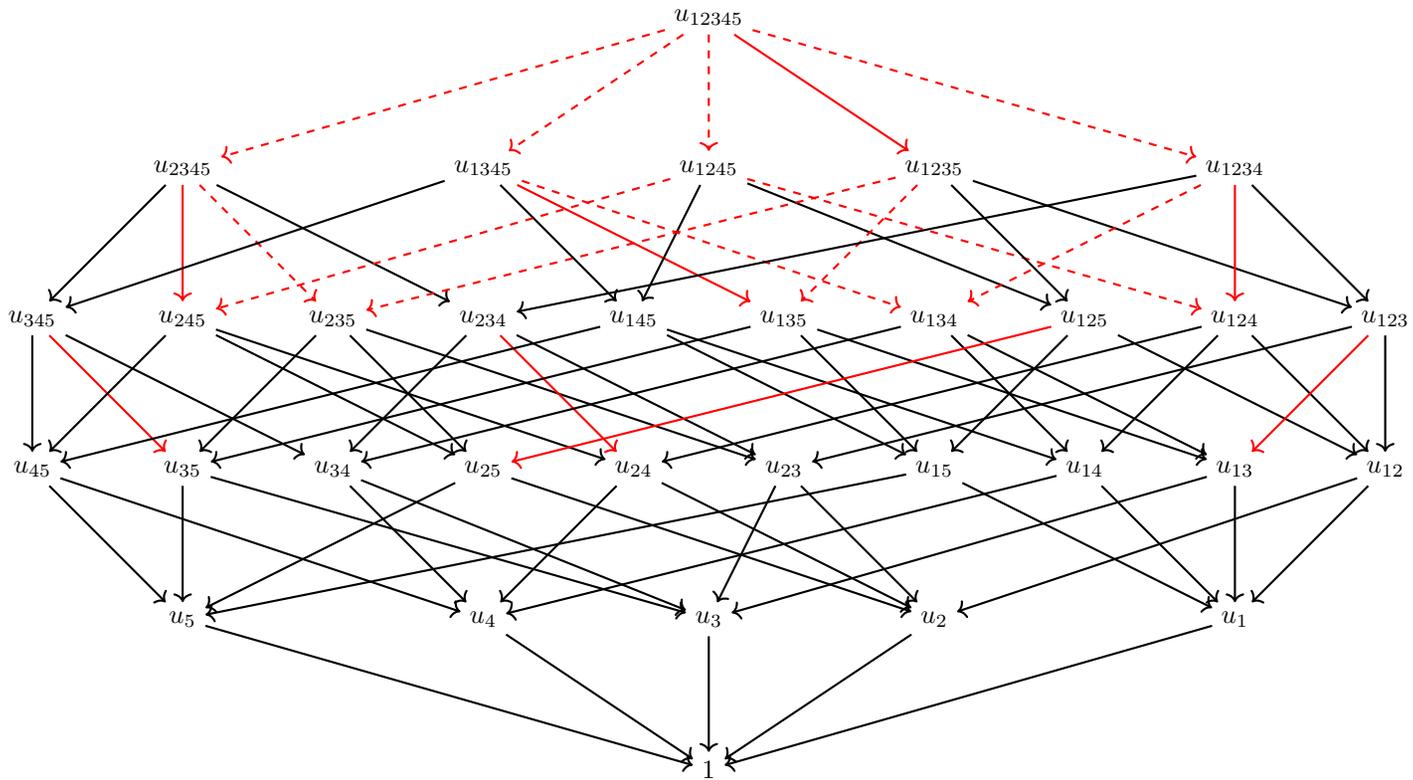

Given a Morse matching $\m$ on a based complex $K$, our next goal is to define a map $\phi\colon K \to K$ of degree $1$ and show that it is a splitting homotopy in the sense of \cite{barneslambe1991}. 
We recall the following definition from \cite{barneslambe1991}.

\begin{definition}
 Let $K$ be a chain complex and $\phi\colon K \to K$ a degree $1$ map. Then $\phi$ is called a \emph{splitting homotopy} if 
\begin{equation*}
\begin{split}
\phi^2 &= 0,\\
\phi d \phi &= \phi. 
\end{split}
\end{equation*}
\end{definition}

Fix a Morse matching $\m$ and write $K_n = \bigoplus_{\alpha \in I_n} K_{\alpha}$. Define a relation $\prec$ on $I_n$ by setting $\alpha \prec \beta$ if there is a directed path from $\alpha$ to $\beta$ in $G_K^{\m}$. Note that since $G_K^{\m}$ does not contain any directed cycles, the relation $\prec$ is a well-founded partial order. Define $\phi$ by induction on $\prec$ as follows. If $\alpha$ is minimal with respect to $\prec$ and $x \in K_{\alpha}$, put
$$ \phi(x) = \begin{cases} d^{-1}_{\alpha,\beta}(x) &\mbox{ if } \beta \to \alpha \in \m \mbox{ for some } \beta, \\ 0 &\mbox{ otherwise. } \end{cases}$$
If $\alpha$ is not minimal with respect to $\prec$ and $x \in K_{\alpha}$, put
$$ \phi(x) = \begin{cases} d^{-1}_{\alpha,\beta}(x) - \sum_{\substack{\beta \to \gamma \\ \gamma \neq \alpha}} \phi d_{\gamma,\beta} d_{\alpha,\beta}^{-1}(x) &\mbox{ if } \beta \to \alpha \in \m \mbox{ for some } \beta, \\ 0 &\mbox{ otherwise. } \end{cases}$$
Note that for all $\gamma$ in the last sum we have $\gamma \prec \alpha$ and so $\phi$ is well-defined. 
Observe that the second definition of $\phi$ is only relevant if $G_K^{\m}$ has a subgraph of the form

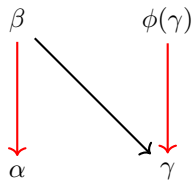
\begin{center}
\begin{tikzpicture}[descr/.style={fill=white,inner sep=1.5pt}]

\node (1) at (0,0) {$\alpha$};
\node (2) at (2,0) {$\gamma$};
\node (3) at (0,2) {$\beta$};
\node (4) at (2,2) {$\phi(\gamma)$};

\path [->] (3) edge[thick,red] (1);
\path [->] (4) edge[thick,red] (2);
\path [->] (3) edge[thick] (2);

\end{tikzpicture}   
\end{center} 
where the red arrows are elements of the matching $\m$. 
\begin{lemma}[\cite{skoldberg2006}, Lemma 2]
Let $\m$ be a Morse matching on a based complex $K$. Then the map $\phi$ is a splitting homotopy. \end{lemma}

Define a map $p\colon K \to K$ by $p = 1_K - (\phi d + d \phi)$. A direct computation show that $p$ is a chain map satisfying $p^2 = p$. Therefore, we have a splitting of chain complexes
$$ K = \ker(p) \oplus \im(p).$$
Let $L = \im(p)$. We have the following lemma.


\begin{lemma}
\label{amttransferdiagram}
There exists a transfer diagram
 \begin{center}
\begin{tikzpicture}
\matrix(m)[matrix of math nodes,
row sep=3em, column sep=2.8em,
text height=1.5ex, text depth=0.25ex]
{L &K\\};
\path[->]
(m-1-1) edge [bend left=35] node[yshift=1.5ex] {$i$} (m-1-2)
(m-1-2) edge [bend left=35] node[yshift=-1.5ex] {$p$} (m-1-1)
(m-1-2) edge [loop right, in=35,out=-35,looseness=5, min distance=10mm] node {$\phi$} (m-1-2)
;
\end{tikzpicture}
\end{center}
where $i$ is the inclusion. 
\end{lemma}
\begin{proof}
We first show that $\phi i = 0$. Indeed, we have
$$\phi p = \phi(1_K - (\phi d + d \phi)) = \phi - \phi^2 d - d \phi d = \phi - 0 - \phi = 0.$$
Since $i$ is a chain map, it follows that
$$\phi d i = \phi i d = 0.$$
Therefore,
$$pi = (1 - \phi d - d \phi)i = 1.$$
By definition of $p$, we have $ip \simeq 1_L$ which finishes the proof. 
\end{proof}

The following theorem is one of the central results of algebraic Morse theory. 

\begin{theorem}[\cite{skoldberg2006}, Theorem 1]
\label{amtmaintheorem}
 Let $\m$ be a Morse matching on a based complex $K$. Then the complexes $K$ and $p(K)$ are homotopy equivalent. Furthermore, the map
\begin{equation}
\label{amtmaintheoremisomorphism}
p\colon \bigoplus_{\alpha \in M_n^0} K_{\alpha} \to L_n
\end{equation}
is an isomorphism of modules for every $n \in \N$. 
\end{theorem}

Note that in general the isomorphism (\ref{amtmaintheoremisomorphism}) is only an isomorphism of graded modules and not of chain complexes. In case the components corresponding to the critical vertices do form a subcomplex we have the following corollary.

\begin{corollary}[\cite{skoldberg2006}, Corollary 2]
\label{skoldbergcor2}
Suppose that $\m$ is a Morse matching on $K$ such that 
$$C = \bigoplus_{\alpha \in\m_n^0} K_{\alpha}$$
is a subcomplex of $K$. Then $K$ and $C$ are homotopy equivalent.  
\end{corollary}

It follows from Theorem \ref{amtmaintheorem} that $C$ admits a differential $\tilde{d}$ such that $(C, \tilde{d})$ is isomorphic to $(L,d)$. Indeed, define
$$q\colon K \to C$$
on $x \in K_{\alpha}$ by
$$q(x) = \begin{cases} x &\mbox{ if } \alpha \in \m_0 \\ 0 &\mbox{ otherwise. } \end{cases}$$
The map $q$ is called the \emph{projection on the critical cells}. Next, $\tilde{d}$ by
$$\tilde{d} = q(d - d\phi d).$$
Then we have the following theorem.

\begin{theorem}[\cite{skoldberg2006}, Theorem 2]
\label{amtmaintheorem2}
The complex $(C, \tilde{d})$ is homotopy equivalent to $(K,d)$. 
\end{theorem}

\begin{definition}
\label{definitionmorsecomplex}
Let $K$ be a complex and $\m$ a Morse matching on $K$. Let $C$ be as in Corollary \ref{skoldbergcor2}. The complex $(C, \tilde{d})$ is called the \emph{Morse complex} of $K$ associated to $\m$. Given a Morse matching $\m$ on $K$, we will write $G^{\m}$ for the Morse complex of $K$ associated to $\m$. 
\end{definition}

\section{$A_{\infty}$-resolutions via algebraic Morse theory}

In this section we will investigate how algebraic Morse theory gives rise to $A_{\infty}$-structures . 
\begin{theorem}
 Let $A$ be a differential graded algebra and let $\m$ be a Morse matching on $A$. Let $\phi$ and $p$ be as before and set $B = \im(p)$. Then $B$ has the structure of an $A_{\infty}$-algebra.
\end{theorem}
\begin{proof}
By Lemma \ref{amttransferdiagram} there exists a transfer diagram
 \begin{center}
\begin{tikzpicture}
\matrix(m)[matrix of math nodes,
row sep=3em, column sep=2.8em,
text height=1.5ex, text depth=0.25ex]
{B &A\\};
\path[->]
(m-1-1) edge [bend left=35] node[yshift=1.5ex] {$i$} (m-1-2)
(m-1-2) edge [bend left=35] node[yshift=-1.5ex] {$p$} (m-1-1)
(m-1-2) edge [loop right, in=35,out=-35,looseness=5, min distance=10mm] node {$\phi$} (m-1-2)
;
\end{tikzpicture}
\end{center}
where $i$ is the inclusion. Therefore, the result follows from Theorem \ref{merkulovtheorem}.
\end{proof}

Given two Morse matchings $\m_1$ and $\m_2$ on a dg algebra $A$, we want to know how the corresponding $A_{\infty}$-structures $\mu_n^1$ and $\mu_n^2$ are related. The main ingredient is the following theorem.

\begin{theorem}[\cite{kontsevichsoibelman2009}]
\label{kontsevichsoibelman}
 Let $f\colon (V,d_V) \to (W,d_W)$ be a chain homotopy equivalence. Then any $A_{\infty}$ structure on $W$ transfers to an $A_{\infty}$-structure on $V$ such that $f$ extends to an $A_{\infty}$-morphism with $f_1=f$ which is an $A_{\infty}$-homotopy equivalence.
\end{theorem}

We have the following result.

\begin{theorem}
\label{morsematchinghomotopyequivalentainfinity}
 Let $A$ be a dg algebra and $\m_1$ and $\m_2$ two Morse matchings on $A$. Put $A_i = \im p_i$ and denote by $\mu_n^i$ the corresponding $A_{\infty}$-structure. 
 Then there exists an $A_{\infty}$-homotopy equivalence
 $$ f\colon (A_1,\mu_n^1) \to (A_2,\mu_n^2).$$
\end{theorem}
\begin{proof}
 Consider the following commutative diagram
\begin{center}
\begin{tikzcd}[row sep=large, column sep=large]
A \arrow{r}{1} \arrow[xshift=-0.7ex,swap]{d}{p_1} & A \arrow[xshift=-0.7ex,swap]{d}{p_2} \\
A_1 \arrow[xshift=0.7ex,swap]{u}{i_1} \arrow[yshift=0.7ex]{r}{p_2i_1} & A_2. \arrow[xshift=0.7ex,swap]{u}{i_2} \arrow[yshift=-0.7ex]{l}{p_1i_2}
\end{tikzcd}
\end{center} 
Here, $p_i$ is defined as
$$ p_i = 1 - d \phi_i - \phi_i d $$
where $\phi$ is the splitting homotopy associated to $\m_i$. By definition, we have
$$ (p_2i_1)(p_1i_2) \simeq p_2i_2 = 1_{A_2} $$
since $i_1p_1$ is chain homotopic to the identity. Similarly, $(p_1i_2)(p_2i_1) \simeq 1_{A_1}$. 
Therefore, $p_2i_1$ is a chain homotopy equivalence and so the result follows from Theorem \ref{kontsevichsoibelman}. 
\end{proof}

\begin{remark}
\label{kontsevichsoibelmanremark}
 Note that if $p_2i_1$ is an isomorphism of chain complexes then it extends to an $A_{\infty}$-isomorphism by a similar argument. 
\end{remark}

Before we proceed it will be instructive to look at a fully worked example. For this purpose, let $S = k[x_1,x_2,x_3,x_4]$ and let $I$ be the ideal generated by 
\begin{center}
\begin{tikzpicture}
\matrix(m)[matrix of math nodes,
row sep=3em, column sep=2em,
text height=1.5ex, text depth=0.25ex]
{u_1 = x_1x_2, &u_2=x_2x_3, &u_3=x_2x_4, &u_4=x_1x_4.\\};
\end{tikzpicture}
\end{center}
That is to say, $R = S/I$ is the Stanley-Reisner ring of the following simplicial complex.
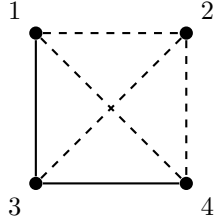
\begin{center}
\begin{tikzpicture}[descr/.style={fill=white,inner sep=1.5pt}]

\node[circle, inner sep=0pt, minimum size=5pt, label=above left:$1$, fill=black] (1) at (0,2) {};
\node[circle, inner sep=0pt, minimum size=5pt, label=above right:$2$, fill=black] (2) at (2,2) {};
\node[circle, inner sep=0pt, minimum size=5pt, label=below left:$3$, fill=black] (3) at (0,0) {};
\node[circle, inner sep=0pt, minimum size=5pt, label=below right:$4$, fill=black] (4) at (2,0) {};

\path (1) edge[thick] (3);
\path (3) edge[thick] (4);
\path (1) edge[dashed,thick] (2);
\path (1) edge[dashed,thick] (4);
\path (2) edge[dashed,thick] (3);
\path (2) edge[dashed,thick] (4);

\end{tikzpicture}   
\end{center}
Let $T$ denote the Taylor resolution of $I$. Figure \ref{matchinggraphfourgenerators} shows the graph $G_T$. 
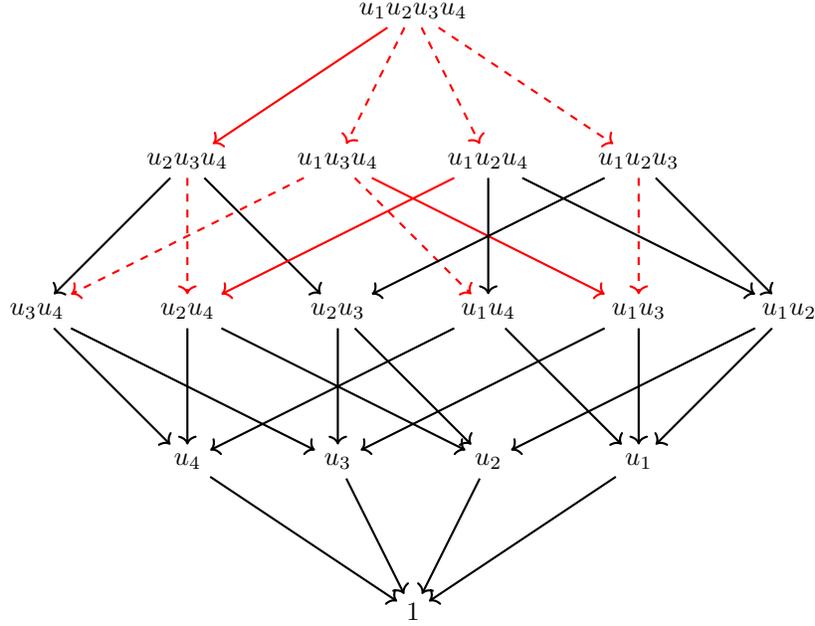
\begin{figure}
\begin{center}
\begin{tikzpicture}[descr/.style={fill=white,inner sep=1.5pt}]

\node (u3u4) at (0,0) {$u_3u_4$};
\node (u2u4) at (2,0) {$u_2u_4$};
\node (u2u3) at (4,0) {$u_2u_3$};
\node (u1u4) at (6,0) {$u_1u_4$};
\node (u1u3) at (8,0) {$u_1u_3$};
\node (u1u2) at (10,0) {$u_1u_2$};

\node (u2u3u4) at (2,2) {$u_2u_3u_4$};
\node (u1u3u4) at (4,2) {$u_1u_3u_4$};
\node (u1u2u4) at (6,2) {$u_1u_2u_4$};
\node (u1u2u3) at (8,2) {$u_1u_2u_3$};

\node (u1u2u3u4) at (5,4) {$u_1u_2u_3u_4$};

\node (u4) at (2,-2) {$u_4$};
\node (u3) at (4,-2) {$u_3$};
\node (u2) at (6,-2) {$u_2$};
\node (u1) at (8,-2) {$u_1$};

\node (1) at (5,-4) {$1$};

\path [->] (u1u2u3u4) edge[thick,red] (u2u3u4);
\path [->] (u1u2u3u4) edge[thick, dashed, red] (u1u3u4);
\path [->] (u1u2u3u4) edge[thick, dashed,red] (u1u2u4);
\path [->] (u1u2u3u4) edge[thick, dashed,red] (u1u2u3);

\path [->] (u2u3u4) edge[thick] (u3u4);
\path [->] (u2u3u4) edge[thick, dashed,red] (u2u4);
\path [->] (u2u3u4) edge[thick] (u2u3);

\path [->] (u1u3u4) edge[thick,dashed,red] (u3u4);
\path [->] (u1u3u4) edge[thick,dashed,red] (u1u4);
\path [->] (u1u3u4) edge[thick,red] (u1u3);

\path [->] (u1u2u4) edge[thick,red] (u2u4);
\path [->] (u1u2u4) edge[thick] (u1u4);
\path [->] (u1u2u4) edge[thick] (u1u2);

\path [->] (u1u2u3) edge[thick] (u2u3);
\path [->] (u1u2u3) edge[thick, dashed, red] (u1u3);
\path [->] (u1u2u3) edge[thick] (u1u2);

\path [->] (u3u4) edge[thick] (u4);
\path [->] (u3u4) edge[thick] (u3);

\path [->] (u2u4) edge[thick] (u4);
\path [->] (u2u4) edge[thick] (u2);

\path [->] (u2u3) edge[thick] (u3);
\path [->] (u2u3) edge[thick] (u2);

\path [->] (u1u4) edge[thick] (u4);
\path [->] (u1u4) edge[thick] (u1);

\path [->] (u1u3) edge[thick] (u3);
\path [->] (u1u3) edge[thick] (u1);

\path [->] (u1u2) edge[thick] (u2);
\path [->] (u1u2) edge[thick] (u1);

\path [->] (u1) edge[thick] (1);
\path [->] (u2) edge[thick] (1);
\path [->] (u3) edge[thick] (1);
\path [->] (u4) edge[thick] (1);

\end{tikzpicture}   
\end{center} 
\caption{The graph $G_T$ corresponding to the ideal $I$.}
\label{matchinggraphfourgenerators}
\end{figure}
In Figure \ref{matchinggraphfourgenerators} the red arrows are those for which $d_{\alpha,\beta}$ is an isomorphism and, hence, which are allowed to be in a Morse matching. The solid red arrows give a specific example of a Morse matching.  
Computing $\phi$ we obtain 
$$ \begin{cases} \phi(u_2u_4) = u_1u_2u_4, \\ \phi(u_1u_3) = u_1u_3u_4, \\ \phi(u_2u_3u_4) = u_1u_2u_3u_4. \end{cases}$$
In all other cases, $\phi$ is zero. For $p$, we compute
\begin{equation*}
p(u_2u_4) = (1-d\phi - \phi d)(u_2u_4)  = u_2u_4 - d(u_1u_2u_4) = x_3u_1u_4 - x_4u_1u_2
\end{equation*}
and
\begin{equation*}
p(u_1u_3) = (1-d\phi - \phi d)(u_1u_3) = u_1u_3 - d(u_1u_3u_4) = u_1u_4 - u_3u_4.
\end{equation*}
Next, we have
\begin{equation*}
\begin{split}
p(u_2u_3u_4) &= (1-d\phi - \phi d)(u_2u_3u_4) \\
&=u_2u_3u_4 - d(u_1u_2u_3u_4) - x_3\phi(u_3u_4) + \phi(u_2u_4) - x_1\phi(u_2u_3)\\
&=u_2u_3u_4 - u_2u_3u_4 + x_3u_1u_3u_4 - u_1u_2u_4 + u_1u_2u_3 + u_1u_2u_4 \\
&=u_1u_2u_3 + x_3u_1u_3u_4.
\end{split}
\end{equation*}
Further, we have $p(u_i) = u_i$ and $p(u_1u_2u_3u_4) = 0$. Also, we have
$$ \begin{cases} p(u_3u_4) = u_3u_4, \\ p(u_2u_3) = u_2u_3, \\ p(u_1u_4) = u_1u_4, \\ p(u_1u_2) = u_1u_2 \end{cases}$$
and
$$ \begin{cases} p(u_1u_3u_4) = 0, \\ p(u_1u_2u_4) = u_1u_2u_4 - \phi(u_2u_4 - x_3u_1u_4 + x_4u_1u_2) =  0. \end{cases}$$
Lastly, we have
\begin{equation*}
\begin{split}
p(u_1u_2u_3) & = (1- d \phi - \phi d)(u_1u_2u_3) \\
&= u_1u_2u_3 - x_1\phi(u_2u_3) + x_3\phi(u_1u_3) - x_4\phi(u_1u_2) \\
&= u_1u_2u_3 + x_3u_1u_3u_4.
\end{split}
\end{equation*}
Consequently, $\im(p)$ is equal to
\begin{center}
\begin{tikzpicture}[descr/.style={fill=white,inner sep=1.5pt}]
        \matrix (m) [
            matrix of math nodes,
            row sep=3em,
            column sep=2em,
            text height=1.5ex, text depth=0.25ex
        ]
        { 0 & S & S^4 & S^4 & S & 0\\
         };

        \path[overlay,->, font=\small, >=latex]
                        (m-1-1) edge (m-1-2) 
                        (m-1-2) edge node[yshift=1.5ex] {$d$} (m-1-3) 
                        (m-1-3) edge node[yshift=1.5ex] {$d$} (m-1-4) 
                        (m-1-4) edge node[yshift=1.5ex] {$d$} (m-1-5) 
                        (m-1-5) edge (m-1-6) 
                        ;
\end{tikzpicture}   
\end{center} 
with basis
\begin{center}
 \begin{tabular}{|c|c|} \hline
  degree &generators \\ \hline
  $0$ &$1$ \\
  $1$ &$u_1$, $u_2$, $u_3$, $u_4$ \\
  $2$ &$u_1u_2$, $u_1u_4$, $u_2u_3$, $u_3u_4$ \\
  $3$ &$u_1u_2u_3 + x_3u_1u_3u_4$ \\ \hline 
 \end{tabular}
\end{center}
Figure \ref{multiplicationtableforfullyworkedexample} depicts the full table for the multiplication $\mu_2=p\lambda_2$ where $y = u_1u_2u_3 + x_3 u_1u_3u_4$. 

\begin{figure}
\begin{center}
  \begin{tabular}{|c|cccc|cccc|} \hline
   &$u_1$ &$u_2$ &$u_3$ &$u_4$ &$u_{12}$ &$u_{14}$ &$u_{23}$ &$u_{34}$ \\ \hline  
   $u_1$ &$0$ & $x_2u_{12}$ & $x_3u_{14}- x_2u_{34}$ & $x_1u_{14}$ & $0$ & $0$ & $x_2y$ & $0$ \\
   $u_2$ & &$0$ & $x_3u_{23}$ & $x_3u_{14} - x_4u_{12}$ & $0$ & $0$ & $0$ & $x_2y$ \\
   $u_3$ & & &$0$ & $x_4u_{34}$ & $x_2y$ & $0$ & $0$ & $0$ \\
   $u_4$ & & & &$0$ & $0$ & $0$ & $x_4y$ & $0$ \\ \hline
   $u_{12}$ & & & & &$0$ &$0$ &$0$ &$0$ \\
   $u_{14}$ & & & & & &$0$ &$0$ &$0$ \\
   $u_{23}$ & & & & & & &$0$ &$0$ \\
   $u_{34}$ & & & & & & & &$0$ \\ \hline
  \end{tabular}
\end{center}
\caption{The multiplication $\mu_2$ for the ideal $I$.}
\label{multiplicationtableforfullyworkedexample}
\end{figure}

We have seen how Morse matchings give rise to $A_{\infty}$-structures on $\im(p)$. Our next goal is to describe $A_{\infty}$-structure on the actual Morse complex. This will allow us to study these structures in terms of the critical vertices of the Morse matching. 

Let $A$ be a differential graded algebra and let $\m$ be a Morse matching on $A$ with corresponding splitting homotopy $\phi$. 
Define
$$ p = 1 - d \phi - \phi d $$
as before. We have seen that there is a transfer diagram
\begin{center}
\begin{tikzpicture}
\matrix(m)[matrix of math nodes,
row sep=3em, column sep=2.8em,
text height=1.5ex, text depth=0.25ex]
{\im(p) &A\\};
\path[->]
(m-1-1) edge [bend left=35] node[yshift=1.5ex] {$i$} (m-1-2)
(m-1-2) edge [bend left=35] node[yshift=-1.5ex] {$p$} (m-1-1)
(m-1-2) edge [loop right, in=35,out=-35,looseness=5, min distance=10mm] node {$\phi$} (m-1-2)
;
\end{tikzpicture}
\end{center}
where $i$ is the inclusion. Consequently, the multiplication $\lambda \colon A^2 \to A$ induces an $A_{\infty}$-structure $\mu_n$ on $\im(p)$ via the Merkulov construction from Theorem \ref{merkulovtheorem}.
Let $A^{\m}$ denote the Morse complex. Then we have a diagram
\begin{center}
\begin{tikzcd}[row sep=large, column sep=large]
A \arrow{r}{1} \arrow[xshift=-0.7ex,swap]{d}{p} & A \arrow[xshift=-0.7ex,swap]{d}{q} \\
\im(p) \arrow[xshift=0.7ex,swap]{u}{i} \arrow[yshift=0.7ex]{r}{qi} & A^{\m} \arrow[xshift=0.7ex,swap]{u}{j} \arrow[yshift=-0.7ex]{l}{pj}
\end{tikzcd}
\end{center} 
where $q$ is the projection on the critical cells. Recall from the proof of Theorem \ref{amtmaintheorem2} that the differential $\tilde{d}$ on $A^{\m}$ can be rewritten as
\begin{equation*}
 \begin{split}
  \tilde{d} &= q(d-d \phi d) j \\
  &= qidpj.
 \end{split}
\end{equation*}
By Theorem 2 of \cite{skoldberg2006}, it follows that $p_1i_2$ is an isomorphism and hence we get the following corollary of Theorem \ref{kontsevichsoibelman}.

\begin{corollary}
\label{ainftystructureonmorsecomplex}
 Let $A$ be a dg algebra and $\m$ a Morse matching. Then the Morse complex $A^{\m}$ has an $A_{\infty}$-algebra structure $\nu_n$ such that there exists an isomorphism of $A_{\infty}$-algebras
 $$ (A^{\m}, \nu_n) \to (\im(p),\mu_n). $$
\end{corollary}
\begin{proof}
 Indeed, define $\nu_n$ by
 $$ \nu_n = qi\mu_n(pj)^{\otimes n}. $$
 The required isomorphism is then the one induced by $p_1i_2$ via Theorem \ref{kontsevichsoibelman}.
\end{proof}

\section{Morse theory on the Taylor resolution}
In this section, we apply algebraic Morse theory to the Taylor resolution $T$. First, we discuss one way of constructing Morse matchings on the Taylor resolution which is due to J\"ollenbeck \cite{jollenbeck2006}.

Given a basis element $u = u_{i_1}\cdots u_{i_p} \in T$, define an equivalence relation as follows.
We say that $u_{i_j}$ and $u_{i_k}$ are equivalent if $\gcd(m_{i_j},m_{i_k}) \neq 1$. In that case, we write $u_{i_j} \sim u_{i_k}$.
The transitive closure of $\sim$ gives an equivalence relation on $u$ and we write $\cl(u)$ for the number of equivalence classes. 
An arrow $u \to v$ in $G_T$ is called \emph{admissible} if $m_u = m_v$ and the Taylor differential $d$ maps $u$ to $v$ with nonzero coefficient.

\begin{construction}
\begin{enumerate}
\item Let $u \to v$ be an admissible arrow with $\cl(u)=\cl(v)=1$ such that no proper subsets $u' \subset u$ and $v' \subset v$ define an admissible arrow $u' \to v'$ with $\cl(u)=1$ and $\cl(v')=1$. 
Define
$$ \m_{11} = \{ uw \to vw \mid \text{for each } w \text{ with } \gcd(m_w,m_u) =1= \gcd(m_w,m_v) \}.$$
To simplify notation, write $u \in \m_{11}$ if there exists $v$ such that either $u \to v$ or $v \to u$ is in $\m_{11}$. 
Then $\m_{11}$ is an acyclic matching. Note that if $\gcd(m_u,m_v) = 1$ and $uv \in \m_{11}$ then $u \in \m_{11}$ or $v \in \m_{11}$. 
Consequently, the same procedure can repeated on the Morse complex  $T^{\m_{11}}$. Therefore, we obtain a series of acyclic matchings $\m_{1} = \cup_{i \geq 1} \m_{1i}$. After finitely many steps we obtain a complex such that for each admissible arrow $u \to v$ we have $\cl(u) \geq 1$ and $\cl(v) \geq 2$.
\item Let $u \to v$ be an admissible arrow in $T^{\m_1}$ with $\cl(u) =1$ and $\cl(v)$ such that no proper subsets $u' \subset u$ and $v' \subset v$ define an admissible arrow $u' \to v'$ with $\cl(u)=1$ and $\cl(v')=2$. 
Define
$$ \m_{21} = \{ uw \to vw \mid \text{for each } w \text{ with } \gcd(m_w,m_u) =1= \gcd(m_w,m_v) \}.$$
By the same argument as before, this procedure can be repeated on the Morse complex $T^{\m_{21}}$. Consequently, we obtain a sequence of acyclic matchings $\m_2 = \cup_{i \geq 1} \m_{2i}$.
\item Continuing on we obtain a sequence of matching $\m = \cup_{i \geq 1} \m_i$.
\end{enumerate}
\end{construction}
Each admissible arrow is of the form $uw \to vw$ where $m_u = m_v$, $\gcd(m_u,m_w)=1$, $\cl(u) = 1$ and $\cl(v) \geq 1$. Therefore, $(T^{\m},\tilde{d})$ is the minimal free resolution of $S/I$. The following lemma is immediate from the above construction.

\begin{lemma}
Let $\m$ be constructed as above. Then
 \begin{enumerate}
  \item for all arrows $u \to v$ in $\m$ we have $m_u = m_v$,
  \item for all arrows in the Morse complex $T^{\m}$ we have $m_u \neq m_v$,
  \item $\m_i$ is a sequence of acyclic matchings on the Morse complex $T^{\m_{<i}}$ where $\m_{<i} = \cup_{j < i} \m_j$,
  \item for all arrows $u \to v$ we have $\cl(u) - \cl(v) = i -1$ and $\vert v \vert +1= \vert u \vert$,
 \end{enumerate}
\end{lemma}

The following lemma is straightforward but will be used often.

\begin{lemma}
\label{morsematchingmaximaliffresolutionminimal}
Let $\m$ be a Morse matching on the Taylor resolution $T$. Then $\m$ is maximal if and only if $T^{\m}$ is the minimal free resolution.
\end{lemma}
\begin{proof}
Clearly, if $T^{\m}$ is minimal then $\m$ cannot be extended and hence is maximal. 

For the converse, if $T^{\m}$ is not minimal then there exists some component $d_{\alpha,\beta}$ which does not map into $(x_1,\ldots,x_m)T$. By definition of the Taylor differential this is only possible if $d_{\alpha,\beta} = \pm 1$ and so is invertible. Define $\m'$ by setting
$$\m' = \m \cup \{ \alpha \to \beta \}.$$
Then it is easily seen that $\m'$ is a Morse matching. Therefore, $\m$ is not maximal. 
\end{proof}

\begin{corollary}
\label{minimalfreeresolutionalwaysfrommorsematching}
Let $R$ be a monomial ring and let $F \to R$ be the minimal free resolution of $R$. Then there exists a maximal Morse matching $\m$ on the Taylor resolution $T$ such that $T^{\m} = F$. 
\end{corollary}
\begin{proof}
Let $\m$ be the matching obtained from J\"ollenbeck's construction. Then $T^{\m}$ is maximal by Lemma \ref{morsematchingmaximaliffresolutionminimal}. Since the minimal free resolution is unique, we have $T^{\m} = F$.
\end{proof}

\begin{theorem}
Let $\m_1$ and $\m_2$ be maximal Morse matchings on $T$ and let $\mu_n^1$ and $\mu_n^2$ be the corresponding $A_{\infty}$-algebra structures. Then there is an $A_{\infty}$-isomorphism
$$(T^{\m_1},\mu_n^1) \cong (T^{\m_2},\mu_n^2).$$ 
\end{theorem}
\begin{proof}
Since the minimal free resolution is unique, it follows from \ref{morsematchingmaximaliffresolutionminimal} that $T^{\m_1} = T^{\m_2}$. Therefore, we can apply Theorem \ref{morsematchinghomotopyequivalentainfinity} and Remark \ref{kontsevichsoibelmanremark} and we get an $A_{\infty}$-isomorphism
$$(T^{\m_1},\mu_n^1) \cong (T^{\m_2},\mu_n^2)$$
which finishes the proof.
\end{proof}

\section{The Golod property}
A well-known result by Berglund and J\"ollenbeck is the following.

\begin{theorem}[\cite{berglundjollenbeck2007}, Theorem 5.1]
\label{berglundjollencktheorem}
Let $R=S/I$ be a monomial ring. Then $R$ is Golod if and only if the product on the Koszul homology $\tor^S(R,k)$ vanishes.  
\end{theorem}
However, in \cite{katthan2015} Katth\"an presented the following counterexample to Theorem \ref{berglundjollencktheorem}.

\begin{example}[\cite{katthan2015}, Theorem 3.1]
Let $k$ be a field and let $S=k[x_1,x_2,y_1,y_2,z]$. Let $I$ be the ideal generated by
\begin{center}
\begin{tikzpicture}
\matrix(m)[matrix of math nodes,
row sep=0.25em, column sep=4em,
text height=1.5ex, text depth=0.25ex]
{m_1 = x_1x_2^2 & m_4 = x_1x_2y_1y_2 & m_7 = x_1y_1z \\
m_2 = y_1y_2^2 & m_5 = y_2^2z^2 & m_8 = x_2^2y_2^2z \\
m_3 = z^3 & m_6 = x_2^2z^2 \\};
\end{tikzpicture}
\end{center}
Then the product on $\tor^S(R,k)$ is trivial but $R$ is not Golod. More precisely, the Massey product $\langle m_1, m_2, m_3 \rangle$ is non-trivial. 
\end{example}

A natural question to ask, then, is under what additional assumptions Theorem \ref{berglundjollencktheorem} does hold. The main purpose of this section is to provide an answer to this question. Recall from Definition \ref{definitionsimplicialresolution} that a resolution $F$ is called simplicial if $F=F_{\Delta}$ for some simplicial complex $\Delta$. 

\begin{definition}
\label{definitionsimpliciallyresolvable}
 A monomial ring $R$ is called \emph{simplicially resolvable} if the minimal free resolution of $R$ is a simplicial resolution. 
\end{definition}

\begin{lemma}
\label{simpliciallyresolvableiffcriticalcellssimplicialcomplex}
Let $R$ be a monomial ring. Then the following are equivalent.
\begin{enumerate}
\item $R$ is simplicially resolvable
\item The Taylor resolution $T$ of $R$ admits a maximal Morse matching $\m$ such that the set $\m_0$ of $\m$-critical cells forms a simplicial complex. 
\end{enumerate}
\end{lemma}
\begin{proof}
If the second statement holds then $F_{\m_0}=T^{\m}$ is a simplicial resolution of $R$. Since $\m$ is maximal, it follows by Lemma \ref{morsematchingmaximaliffresolutionminimal} that $F_{\m_0}$ is minimal.

Conversely, assume that $R$ is simplicially resolvable. Let $F=F_{\Delta}$ denote the minimal free resolution and let $T$ denote the Taylor resolution. Then there exists a trivial complex $G = \oplus_{\alpha \in I} G_{\alpha}$ where
\begin{center}
\begin{tikzpicture}[descr/.style={fill=white,inner sep=1.5pt}]
\matrix (m) [
	matrix of math nodes,
    row sep=3em,
    column sep=2em,
    text height=1.5ex, text depth=0.25ex
    ]
    { G_{\alpha}: & 0 & Su_{\alpha} & Sv_{\alpha} & 0 \\
    };

    \path[overlay,->, font=\small, >=latex]
    	(m-1-2) edge (m-1-3)
        (m-1-3) edge node[yshift=1.5ex] {$d_{\alpha}$} (m-1-4)
        (m-1-4) edge (m-1-5);
\end{tikzpicture}   
\end{center} 
with $d_{\alpha}(u_{\alpha}) = v_{\alpha}$. Define 
$$\m = \{ d_{\alpha} \colon u_{\alpha} \to v_{\alpha} \mid \alpha \in I\}$$
then $\m$ is Morse matching and $F_{\Delta} = T^{\m_0}$. Therefore, $\m_0 = \Delta$ and hence $\m_0$ is a simplicial complex.  
\end{proof}

The following lemma is straightforward but crucial in what follows. 

\begin{lemma}
\label{ainfinitystructuremorsecomplexsimpliciallyresolvable}
Let $R$ be simplicially resolvable and let $\m$ be a Morse matching on the Taylor resolution $T$ of $R$ as in Lemma \ref{simpliciallyresolvableiffcriticalcellssimplicialcomplex}. Let $\nu_n$ denote the corresponding $A_{\infty}$-structure on the Morse complex $T^{\m}$. Then 
$$\nu_n = q \circ \lambda_n \circ j^{\otimes n}$$
where $q\colon T \to T^{\m}$ is the projection on the critical cells, $j\colon T^{\m} \to T$ is the inclusion and $\lambda_n\colon T^{\otimes n} \to T$ is the auxiliary map \eqref{merkulovlambda} from the Merkulov construction. 
\end{lemma}
\begin{proof}
Recall from Corollary \ref{ainftystructureonmorsecomplex} that $\nu_n$ is given by
$$\nu_n = qi \mu_n pj = qip \lambda_n (ipj)^{\otimes n}.$$
Let $d$ denote the differential of the Taylor resolution $T$ and let $\tilde{d}$ denote the differential of the Morse complex $\tilde{d}$ from Definition \ref{definitionmorsecomplex}. That is to say, $\tilde{d} = q(d - d\phi d)j$. Since $R$ is simplicially resolvable, it follows by Lemma \ref{simpliciallyresolvableiffcriticalcellssimplicialcomplex} the critical cells $\m_0$ form a simplicial complex. Therefore, if $u \in T$ is a critical then all $v \subseteq u$ are critical as well. Consequently, $d(T^{\m}) \subseteq T^{\m}$. Now, $p=1-d \phi - \phi d$ and $\phi(T^{\m}) = 0$ by definition. Hence, for $u \in T^{\m}$ we have
$$pj(x) = (1-d \phi - \phi d)jx = x - d \phi x - \phi d x = x.$$
Consequently, the isomorphism $pj\colon T^{\m} \to \im(p)$ is just the identity. Hence so is its inverse $qi$. Therefore, 
$$\nu_n = q \circ \lambda_n \circ j^{\otimes n}$$
as required. 
\end{proof}

Next, we investigate how the maps $\lambda_n$ behave with respect to multidegrees. 

\begin{lemma}
 For all $n$ and all $v_1,\ldots,v_n \in T$, $\lcm(\lambda_n(v_1,\ldots,v_n))$ divides $\lcm(v_1,\ldots, v_n)$. 
\end{lemma}
\begin{proof}
 We prove the statement by induction on $n$. If $n=2$, then
$$ \lambda_2(v_1,v_2) = \begin{cases} \frac{\lcm(v_1)\lcm(v_2)}{\lcm(v_1v_2)}v_1v_2 &\mbox{ if } v_1 \cap v_2 = \emptyset \\ 0 &\mbox{ otherwise }\end{cases}$$
and so the result is clear. Next, assume the result holds for all degrees up to $n-1$. Fix some $k+l=n$. 
By assumption, $\lcm(\lambda_k(v_1,\ldots,v_k))$ is a divisor of $\lcm(v_1,\ldots, v_k)$. 
If $\phi \lambda_k(v_1,\ldots,v_k) = 0$ then there is nothing to prove. Otherwise, there exists some $x \to \lambda_k(v_1,\ldots,v_k) \in \m$.
Then we can write
$$ \phi \lambda_k(v_1,\ldots,v_k) = x + \sum_{\alpha} y_{\alpha} $$
for some $\alpha$. If $y_{\alpha} \neq 0$ for some $\alpha$, then in the Morse graph $G_T$ there is a subgraph
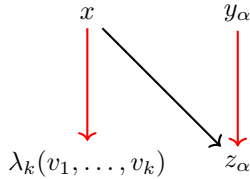
\begin{center}
\begin{tikzpicture}[descr/.style={fill=white,inner sep=1.5pt}]

\node (1) at (0,0) {$\lambda_k(v_1,\ldots,v_k)$};
\node (2) at (2,0) {$z_{\alpha}$};
\node (3) at (0,2) {$x$};
\node (4) at (2,2) {$y_{\alpha}$};

\path [->] (3) edge[thick,red] (1);
\path [->] (4) edge[thick,red] (2);
\path [->] (3) edge[thick] (2);

\end{tikzpicture}   
\end{center} 
where the red arrows are in the Morse matching $\m$. By definition of $\m$, we have $\lcm(x) = \lcm(\lambda_k(v_1,\ldots,v_k))$ and $\lcm(y_{\alpha}) = \lcm(z_{\alpha})$ for all $\alpha$. 
Since $\lcm(z_{\alpha})$ is a divisor of $\lcm(x)$, it follows that $\lcm(y_{\alpha})$ is a divisor or $\lcm(\lambda_k(v_1,\ldots,v_k)) = \lcm(v_1,\ldots,v_k)$. 
The result now follows from the case $n=2$. 
\end{proof}

Our next goal is to give a lower bound for $\cl(\lambda_n(v_1,\ldots,v_n))$. We have the following lemmma.

\begin{lemma}
 Let $n\geq 2$ and let $v_1,\ldots, v_n \in T$ with $\gcd(v_i,v_j)=1$ for $i \neq j$. Then
 $$ \cl(\lambda_n(v_1,\ldots,v_n)) \geq 2. $$
\end{lemma}
\begin{proof}
 Fix some $k+l=n$ and let $v_1,\ldots, v_n \in T$ with $\gcd(v_i,v_j)=1$ for $i \neq j$. It is sufficient to show that
 $$ \gcd \big(\lcm(\phi \lambda_k(v_1,\ldots,v_k)),\lcm(\phi \lambda_l(v_{k+1},\ldots,v_n))\big)=1. $$
 By the previous lemma, it follows that $\lcm(\phi \lambda_k(v_1,\ldots,v_k))$ is a divisor of $\lcm(v_1,\ldots,v_k)$ and that $\lcm(\phi \lambda_l(v_{k+1},\ldots,v_n))$ is a divisor of $\lcm(v_{k+1},\ldots,v_n)$.
Since 
$$ \gcd \big( \lcm(v_1,\ldots,v_k), \lcm(v_{k+1},\ldots,v_n) \big) =1, $$
it follows that $\cl(\lambda_n(v_1,\ldots,v_n)) \geq 2$ as desired. 
\end{proof}

We now come to the first main theorem of this section.

\begin{theorem}
\label{suffcientmorsematchinggolod}
Let $R$ be simplicially resolvable and let $T \to R$ denote the Taylor resolution. Suppose that the Taylor resolution admits a Morse matching $\m$ such that for all $u \in T$ with $\cl(u) \geq 2$ we have $u \in \m$.
 Then $S/I$ is Golod. 
\end{theorem}
\begin{proof}
For a given Morse matching $\m$ we have a diagram
\begin{equation}
\label{morsediagram}
\begin{tikzcd}[row sep=large, column sep=large]
T \arrow{r}{1} \arrow[xshift=-0.7ex,swap]{d}{p} & T \arrow[xshift=-0.7ex,swap]{d}{q} \\
F=\im(p) \arrow[xshift=0.7ex,swap]{u}{i} \arrow[yshift=0.7ex]{r}{f} & T^{\m} \arrow[xshift=0.7ex,swap]{u}{j} \arrow[yshift=-0.7ex]{l}{g}
\end{tikzcd}
\end{equation} 
where $f=qi$ and $g = pj$. As usual, we let $\lambda_k\colon T^{\otimes k} \to T$ denote the auxiliary maps from the Merkulov construction and $\mu_k\colon F^{\otimes k} \to F$ the $A_{\infty}$-structure on $F$. 
That is,
$$ \mu_k = p \circ \lambda_k \circ i^{\otimes k}.$$
Since $R$ is simplicially resolvable, it follows from Lemma \ref{ainfinitystructuremorsecomplexsimpliciallyresolvable} that the maps  
\begin{equation*}
\nu_k = qi \circ \mu_k \circ (pj)^{\otimes k} = q \circ \lambda_k \circ j^{\otimes k}
\end{equation*}
give an $A_{\infty}$-structure on the Morse complex $T^{\m}$. Suppose that the Massey product $\langle u_1,\ldots,u_n \rangle$ is defined. It is sufficient to show that $\nu_n(u_1,\ldots,u_n) \in (x_1,\ldots,x_m)T^{\m}$.
We may assume $u_1,\ldots,u_n$ have pairwise trivial gcd since otherwise the Massey product will be trivial. 
By the previous lemma, 
$$ \cl(\lambda_n(v_1,\ldots,v_n)) \geq 2.$$
 Consequently, 
$$ \nu_n(u_1,\ldots,u_n) = q\lambda_n(v_1,\ldots,v_n) = 0 $$
 since $q$ is the projection on the critical cells.
\end{proof}

In what follows, we will denote the product on $\tor^S(R,k)$ by $\smile$. We have the following lemma. 
\begin{lemma}
Let $R$ be simplicially resolvable. Then the following are equivalent. 
 \begin{enumerate}
  \item The Taylor resolution $T$ admits a Morse matching $\m$ such that for all $u \in T$ with $\cl(u) \geq 2$ we have $u \in \m$. 
  \item The product on Koszul homology is trivial. 
 \end{enumerate}
\end{lemma}
\begin{proof}
First, assume that $T$ admits such a Morse matching and call it $\m$. As before, we obtain a diagram \eqref{morsediagram} with $f=qi$ and $g = pj$. Let $u,v \in F$. It is suffcient to show that $\mu_2(u,v) \in (x_1,\ldots,x_m)F$.
We may assume that $\gcd(m_u,m_v)=1$ since otherwise $\lambda_2(u,v) \in (x_1,\ldots,x_m)F$ and hence $\mu_2(u,v) \in (x_1,\ldots,x_m)F$.
Since $\gcd(m_u,m_v)=1$, it follows that $\lambda_2(u,v) = uv$. Therefore,
$$ \mu_2(u,v) = p(uv) = gq(uv). $$
But since $uv \in \m$ by assumption, we have $q(uv)=0$ hence $\mu_2(u,v)=0$. Therefore, the cup product is trivial. 

For the converse, suppose that the first statement does not hold. Then for every Morse matching $\m$ there is some $u$ with $\cl(u) \geq 2$ such that $u \notin \m$. 
So, fix some $\m$ and pick $u \notin \m$ with $\cl(u) \geq 2$. Since $\cl(u) \geq 2$, there exist $v,w$ such that $u = \lambda_2(v,w)$. Since $R$ is simplicially resolvable, it follows by Lemma \ref{simpliciallyresolvableiffcriticalcellssimplicialcomplex} that $v$ and $w$ are critical.
Note that necessarily $\gcd(m_v,m_w)=1$. 
We claim that 
$$ [v] \smile [w] = [u] \neq 0. $$
Let $\nu_n$ be the $A_{\infty}$-structure on $T^{\m}$ corresponding to $\m$, that is
$$ \nu_n = f \circ \mu_n \circ g^{\otimes n}.$$
Then $\nu_n \otimes 1 = \smile$ as $T^{\m}$ is minimal. Compute
\begin{equation*}
\nu_2(v,w) = f\mu_2(gv,gw) = fp \lambda_2(igv,igw) = q \lambda_2(jv,jw) = q(vw) = vw
\end{equation*}
where the last step follows because $vw$ is $\m$-critical by assumption. Hence $[v] \smile [w] = [u] \neq 0$ as desired. 
\end{proof}

We now come to the second main theorem of this section. Recall that if $\m$ is a Morse matching then we denote by $\m_0$ the set of critical cells. We have the following result.

\begin{theorem}
\label{simpliciallyresolvableimpliesgolodiffproducttrivial}
Let $R=S/I$ be simplicially resolvable. Then the following are equivalent.
\begin{enumerate}
\item $R$ is Golod
\item The product on $\tor^S(R,k)$ is trivial.
\item $I$ satisfies the gcd condition. That is, for any two generators $m_1$ and $m_2$ of $I$ with $\gcd(m_1,m_2)=1$ there exists a generator $m \neq m_1,m_2$ such that $m$ divides $\lcm(m_1,m_2)$.
\item For $u,v \in \m_0$ we have $\lcm(u)\lcm(v) \neq \lcm(uv)$ whenever $uv \in \m_0$. 
\end{enumerate}
\end{theorem}
\begin{proof}
We first prove the equivalence $1 \Leftrightarrow 2$. If $R$ is Golod then the product is trivial by definition. Conversely, if the product is trivial then it follows by the previous lemma that the Taylor resolution $T$ admits a Morse matching $\m$ such that for all $u \in T$ with $\cl(u) \geq 2$ we have $u \in \m$. But this implies that all Massey products vanish by Theorem \ref{suffcientmorsematchinggolod}. 

The equivalence $2 \Leftrightarrow 3$ is well-known, see for example Lemma 2.4 of \cite{katthan2015}. We prove $2 \Leftrightarrow 4$. Since $R$ is simplicially resolvable the product on $\tor^S(R,k)$ is induced $q\lambda_2$. Assume the product is trivial and let $u,v \in \m_0$. Then either $\lambda_2(u,v) = 0$ or $q\lambda_2(u,v) \in (x_1,\ldots,x_m)$. In the first case, $uv \notin \m_0$ by definition. In the second case, we have $uv \in \m_0$ and
$$q \lambda_2(u,v) = \frac{\lcm(u)\lcm(v)}{\lcm(uv)} q(uv).$$
So $q\lambda_2(u,v) \in (x_1,\ldots,x_m)$  implies that $\lcm(u)\lcm(v) \neq \lcm(uv)$.

For the converse implication, let $u,v \in \m_0$. If $uv \notin \m_0$, then $q(uv) = 0$ and so $u \smile v = 0$. So, assume $uv \in \m_0$. Then $\lcm(u)\lcm(v) \neq \lcm(uv)$ and so $q\lambda_2(u,v) \in (x_1,\ldots,x_m)$. Consequently, $u \smile v = 0$. 
\end{proof}

The following examples show that the class of simplicially resolvable is quite expansive.

\begin{example}
Let $I$ be a monomial ideal. Recall that $I$ is called \emph{strongly generic} \cite{bayerpeevasturmfels1998} if no variable $x_i$ occurs with the same nonzero exponent in two distinct minimal generators of $I$. By Theorem 3.2 of \cite{bayerpeevasturmfels1998}, it follows that $S/I$ is simplicially resolvable.   
\end{example}

We point out that in \cite{bayerpeevasturmfels1998} it is claimed that the minimal free resolution of a strongly generic ideal always has a dg algebra structure. However, recently a counterexample to this claim was found in \cite{katthan2016b}.

\begin{example}
For a monomial $m = x_1^{a_1} \cdots x_m^{a_m} \in k[x_1,\ldots,x_m]$, write
$$\supp(m) = \{ i \mid a_i \neq 0 \}$$
for the \emph{support} of $m$. A monomial ideal $I = (m_1,\ldots,m_r)$ is called \emph{generic} \cite{millersturmfelsyanagawa2000} if for any distinct $m_i$ and $m_j$ that have the same positive degree in some variable $x_s$ there exists a third generator $m_k$ such that $m_k$ divides $\lcm(m_i,m_j)$ and
$$\supp \big( \frac{\lcm(m_i,m_j)}{m_k} \big) = \supp(\lcm(m_i,m_j)).$$ 
If $I$ is generic then $S/I$ is simplicially resolvable by Theorem 1.5 of \cite{millersturmfelsyanagawa2000}.
\end{example}

Next, we want to investigate the vanishing on higher Massey products. First, recall the definition of a standard matching introduced in \cite{jollenbeck2006}.

\begin{definition}[\cite{jollenbeck2006}, Definition 3.1]
\label{standardmatching}
 Let $\m = \cup_{i \geq 0} \m_i$ be a sequence of matchings on the Taylor resolution $T$. Then $\m$ is called a \emph{standard matching} if the following hold
 \begin{enumerate}
  \item for all arrows $u \to v$ in $\m$, we have $m_u = m_v$,
  \item for all arrows in the Morse complex $T^{\m}$, we have $m_u \neq m_v$,
  \item $\m_i$ is a sequence of acyclic matchings on the Morse complex $T^{\m_{<i}}$, where $\m_{<i} = \cup_{j < i} \m_j$,
  \item for all arrows $u \to v$ in $\m_i$, we have $\cl(u) - \cl(v) = i -1$ and $\vert v \vert +1= \vert u \vert$,
  \item there exist $\mathcal{B}_i \subset \m_i$ such that
        \begin{enumerate}
         \item $\m_i = \mathcal{B}_i \cup \lbrace u \cup w \to v \cup w \mid \gcd(m_u,m_w)=1 \mbox{ and } u \to v \in \mathcal{B}_i \rbrace$,
         \item for all arrows $u \to v$ in $\mathcal{B}_i$, we have $\cl(u) =1$ and $\cl(v) = i$.
        \end{enumerate}
 \end{enumerate}
\end{definition}

\begin{theorem}
\label{standardmatchingimpliesallhighermasseytrivial}
Let $R$ be simplicially resolvable. Suppose that the Taylor resolution $T$ admits a standard matching. Then all higher Massey products are trivial.
\end{theorem}
\begin{proof}
 Let $\m$ be a standard matching on $T$. We obtain an $A_{\infty}$-structure on the Morse complex $T^{\m}$ by
\begin{equation*}
\nu_k = q \circ \mu_k \circ j^k.
\end{equation*}
Suppose that the Massey product $\langle u_1,\ldots,u_n \rangle$ is defined. It is sufficient to show that $\nu_n(u_1,\ldots,u_n) \in (x_1,\ldots,x_m)T^{\m}$.
We may assume the $u_i$ have pairwise trivial gcd since otherwise the Massey product will be trivial. 
We have
\begin{equation*}
\lambda_n(u_1,\ldots,u_n) = \sum_{k+l=n} (-1)^{k+1} \lambda_2(\phi\lambda_k(u_1,\ldots,u_k), \phi\lambda_l(u_{k+1},\ldots,u_n)).
\end{equation*}
Fix some $k,l$. We may assume that $\phi\lambda_k(u_1,\ldots,u_k) \neq 0$. Therefore, there exists some $x \to \lambda_k(u_1,\ldots,u_k) \in \m$.
Then we can write
$$ \phi\lambda_k(u_1,\ldots,u_k) = x + \sum_{\alpha} y_{\alpha} $$
for some $\alpha$. If $y_{\alpha} \neq 0$ for some $\alpha$, then in the Morse graph $G_T$ there is a subgraph
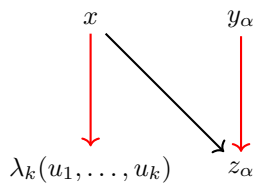
\begin{center}
\begin{tikzpicture}[descr/.style={fill=white,inner sep=1.5pt}]

\node (1) at (0,0) {$\lambda_k(u_1,\ldots,u_k)$};
\node (2) at (2,0) {$z_{\alpha}$};
\node (3) at (0,2) {$x$};
\node (4) at (2,2) {$y_{\alpha}$};

\path [->] (3) edge[thick,red] (1);
\path [->] (4) edge[thick,red] (2);
\path [->] (3) edge[thick] (2);

\end{tikzpicture}   
\end{center} 
where the red arrows are in the Morse matching $\m$. We have that
$$ \lcm(\lambda_k(u_1,\ldots,u_k)) = \lcm(x) $$
and
$$ \lcm(z_{\alpha}) = \lcm(y_{\alpha}). $$
Since $\lcm(z_{\alpha})$ divides $\lcm(x)$, it follows that $\lcm(y_{\alpha})$ divides $\lcm(\lambda_k(u_1,\ldots,u_k))$. 
Therefore, $\phi \lambda_l(u_{k+1},\ldots,u_n)$ is disjoint from $x$ and $y_{\alpha}$. 
By definition of standard matching, it follows that 
$$ x\phi \lambda_l(u_{k+1},\ldots,u_n) \to uv\phi \lambda_l(u_{k+1},\ldots,u_n) \in \m $$
and 
$$ y_{\alpha}\phi \lambda_l(u_{k+1},\ldots,u_n) \to z_{\alpha}\phi \lambda_l(u_{k+1},\ldots,u_n) \in \m. $$
So,
$$ q \lambda_2(\phi\lambda_k(u_1,\ldots,u_k), \phi\lambda_l(u_{k+1},\ldots,u_n)) = 0 $$
since $q$ is the projection on the $\m$-critical cells and elements $x\phi \lambda_l(u_{k+1},\ldots,u_n)$ and $y_{\alpha}\phi \lambda_l(u_{k+1},\ldots,u_n)$ are not $\m$-critical. 
Consequently, 
$$ \nu_n(u_1,\ldots,u_n) = 0 $$
and so the Massey product $\langle u_1, \ldots, u_n \rangle$ is trivial as desired. 
\end{proof}

As said before, in \cite{berglundjollenbeck2007} it is claimed that the Golod property is equivalent to the vanishing of the product on $\tor^S(R,k)$. However, in \cite{katthan2015} a counterexample to this claim is given. The problem is to be found in \cite{jollenbeck2006} where it is claimed that standard matchings always exist. However, the following example due to Katth\"an \cite{katthan2015} shows that this is not the case. 

\begin{example}
Let $S = k[x_1,x_2,x_3,x_4]$ and let $I$ denote the ideal
$$ I = (x_1^2, x_1x_2, x_2x_3, x_3x_4,x_4^2).$$
We will show that $I$ has a non-trivial higher Massey product. In particular, it will then follow by Theorem \ref{standardmatchingimpliesallhighermasseytrivial} that $I$ does not admit a standard matching. For ease of notation, denote the generators of $I$ by
\begin{center}
\begin{tikzpicture}
\matrix(m)[matrix of math nodes,
row sep=3em, column sep=2em,
text height=1.5ex, text depth=0.25ex]
{u_1 = x_1^2, &u_2=x_1x_2, &u_3=x_2x_3, &u_4=x_3x_4, &u_5=x_4^2.\\};
\end{tikzpicture}
\end{center}
Further, we will write $u_A = \prod_{i \in A} u_i$. Define a matching $\m$ by 
\begin{center}
\begin{tikzpicture}
\matrix(m)[matrix of math nodes,
row sep=0.25em, column sep=5em,
text height=1.5ex, text depth=0.25ex]
{ & u_{2345} \to u_{245} &  u_{345} \to u_{35} \\
u_{12345} \to u_{1235} & u_{1345} \to u_{135} & u_{234} \to u_{24} \\
& u_{1234} \to u_{124} & u_{123} \to u_{13} \\};
\end{tikzpicture}
\end{center}
These are the solid red arrows in Figure \ref{avramovsexamplegraph}. This choice of $\m$ gives an acyclic matching satisfying the first four conditions of Definition \ref{standardmatching}. 

The only Massey product that can possibly be nontrivial is $\langle u_1,u_2,u_3 \rangle$ since these are the only disjoint generators. We compute
\begin{equation*}
 \mu_2(u_1, u_3) = p(u_1u_3) = (1- d \phi - \phi d)(u_1u_3) = u_1u_3 + d(u_1u_2u_3) = x_1u_2u_3 + x_3u_1u_2
\end{equation*}
and
\begin{equation*}
 \mu_2(u_3, u_5) = p(u_3u_5) = (1- d \phi - \phi d)(u_3u_5) = u_3u_5 + d(u_3u_4u_5) = x_2u_4u_5 + x_4u_3u_4.
\end{equation*}
Therefore, in $\tor^S(S/I,k)$ both $u_1u_3$ and $u_3u_5$ are zero and so the Massey product $\langle u_1,u_2,u_3 \rangle$ is defined. \\
To get rid of signs, we assume the characteristic of $k$ is two. Then we have
\begin{equation*}
\begin{split}
 \lambda_3(u_1, u_3, u_5) &= \lambda_2(\phi \lambda_1 u_1, \phi \lambda_2(u_3, u_5)) + \lambda_2(\phi \lambda_2(u_1,u_3), \phi \lambda_1 u_5) \\
 &= \lambda_2(u_1, \phi \lambda_2(u_3, u_5)) + \lambda_2(\phi \lambda_2(u_1,u_3), u_5) \\
 &= \lambda_2(u_1, u_3u_4u_5)) + \lambda_2(u_1u_2u_3, u_5) \\
 &= u_1u_3u_4u_5 + u_1u_2u_3u_5.
\end{split}
\end{equation*}
Now,
\begin{equation*}
\begin{split}
 p(u_1u_3u_4u_5) &= (1-d \phi - \phi d)(u_1u_3u_4u_5) \\
 &= u_1u_3u_4u_5 - x_1^2 \phi(u_3u_4u_5) + x_2 \phi(u_1u_4u_5)\\
 &- \phi(u_1u_3u_5) + x_4 \phi(u_1u_3u_4) \\
 &= u_1u_3u_4u_5 - u_1u_3u_4u_5 \\
 &= 0 
\end{split}
\end{equation*}
and
\begin{equation*}
\begin{split}
p(u_1u_2u_3u_5) &= (1 - d \phi - \phi d)(u_1u_2u_3u_5) \\
&=u_1u_2u_3u_5 + d(u_1u_2u_3u_4u_5) - x_1 \phi(u_2u_3u_5) \\
&+ \phi(u_1u_3u_5) - x_3 \phi(u_1u_2u_5) + x_4^2 \phi(u_1u_2u_3) \\
&= u_1u_2u_3u_5 + x_1 u_2u_3u_4u_5 - u_1u_3u_4u_5 + u_1u_2u_4u_5 \\
&-u_1u_2u_3u_5  + x_4 u_1u_2u_3u_4 + u_1u_3u_4u_5 \\
&= x_1 u_2u_3u_4u_5 + u_1u_2u_4u_5 + x_4 u_1u_2u_3u_4.
\end{split}
\end{equation*}
Thus,
$$ \mu_3(u_1,u_3u_5) = x_1 u_2u_3u_4u_5 + u_1u_2u_4u_5 + x_4 u_1u_2u_3u_4 $$
which does not lie in the maximal ideal. 
Next, we show that the indeterminancy of $\langle u_1,u_3,u_5 \rangle$ is zero. Again, it is sufficient to show that
$$(u_1,u_5) \cap \tor^S_4(R,k) = 0.$$
So suppose $u_1v \in \tor^S_4(R,k)$. Since $\mdeg(u_1) = x_1^2$, it follows that $\mdeg(v) = x_2x_3x_4^2$. Since there are no critical cells of multidegree $x_2x_3x_4^2$, we get $v=0$. Therefore, $\langle u_1,u_3,u_5 \rangle = u_1u_2u_4u_5$ and so $S/I$ has a nontrivial Massey product.
\end{example}

\begin{remark}
Since $S/I$ has a non-trivial higher Massey product, it follows that there does not exist a dg algebra structure on the minimal free resolution of $S/I$. Indeed, $S/I$ was the first example of such a monomial ring \cite{avramov1981} but the original proof uses different methods to establish this. 
\end{remark}

\begin{figure}
\begin{center}
\begin{sideways}
\begin{tikzpicture}[descr/.style={fill=white,inner sep=1.5pt}]

\node (45) at (0,0) {$u_{45}$};
\node (35) at (2,0) {$u_{35}$};
\node (34) at (4,0) {$u_{34}$};
\node (25) at (6,0) {$u_{25}$};
\node (24) at (8,0) {$u_{24}$};
\node (23) at (10,0) {$u_{23}$};
\node (15) at (12,0) {$u_{15}$};
\node (14) at (14,0) {$u_{14}$};
\node (13) at (16,0) {$u_{13}$};
\node (12) at (18,0) {$u_{12}$};

\node (345) at (0,2) {$u_{345}$};
\node (245) at (2,2) {$u_{245}$};
\node (235) at (4,2) {$u_{235}$};
\node (234) at (6,2) {$u_{234}$};
\node (145) at (8,2) {$u_{145}$};
\node (135) at (10,2) {$u_{135}$};
\node (134) at (12,2) {$u_{134}$};
\node (125) at (14,2) {$u_{125}$};
\node (124) at (16,2) {$u_{124}$};
\node (123) at (18,2) {$u_{123}$};

\node (2345) at (2,4) {$u_{2345}$};
\node (1345) at (6,4) {$u_{1345}$};
\node (1245) at (9,4) {$u_{1245}$};
\node (1235) at (12,4) {$u_{1235}$};
\node (1234) at (16,4) {$u_{1234}$};

\node (5) at (2,-2) {$u_{5}$};
\node (4) at (6,-2) {$u_{4}$};
\node (3) at (9,-2) {$u_{3}$};
\node (2) at (12,-2) {$u_{2}$};
\node (1) at (16,-2) {$u_{1}$};

\node (12345) at (9,6) {$u_{12345}$};

\node (0) at (9,-4) {$1$};

\path [->] (5) edge[thick] (0);
\path [->] (4) edge[thick] (0);
\path [->] (3) edge[thick] (0);
\path [->] (2) edge[thick] (0);
\path [->] (1) edge[thick] (0);

\path [->] (45) edge[thick] (5);
\path [->] (45) edge[thick] (4);

\path [->] (35) edge[thick] (5);
\path [->] (35) edge[thick] (3);

\path [->] (34) edge[thick] (4);
\path [->] (34) edge[thick] (3);

\path [->] (25) edge[thick] (5);
\path [->] (25) edge[thick] (2);

\path [->] (24) edge[thick] (4);
\path [->] (24) edge[thick] (2);

\path [->] (23) edge[thick] (3);
\path [->] (23) edge[thick] (2);

\path [->] (15) edge[thick] (5);
\path [->] (15) edge[thick] (1);

\path [->] (14) edge[thick] (4);
\path [->] (14) edge[thick] (1);

\path [->] (13) edge[thick] (3);
\path [->] (13) edge[thick] (1);

\path [->] (12) edge[thick] (2);
\path [->] (12) edge[thick] (1);

\path [->] (345) edge[thick] (45);
\path [->] (345) edge[red,thick] (35);
\path [->] (345) edge[thick] (34);

\path [->] (245) edge[thick] (45);
\path [->] (245) edge[thick] (25);
\path [->] (245) edge[thick] (24);

\path [->] (235) edge[thick] (35);
\path [->] (235) edge[thick] (25);
\path [->] (235) edge[thick] (23);

\path [->] (234) edge[thick] (34);
\path [->] (234) edge[red,thick] (24);
\path [->] (234) edge[thick] (23);

\path [->] (145) edge[thick] (45);
\path [->] (145) edge[thick] (15);
\path [->] (145) edge[thick] (14);

\path [->] (135) edge[thick] (35);
\path [->] (135) edge[thick] (15);
\path [->] (135) edge[thick] (13);

\path [->] (134) edge[thick] (34);
\path [->] (134) edge[thick] (14);
\path [->] (134) edge[thick] (13);

\path [->] (125) edge[thick] (25);
\path [->] (125) edge[thick] (15);
\path [->] (125) edge[thick] (12);

\path [->] (124) edge[thick] (24);
\path [->] (124) edge[thick] (14);
\path [->] (124) edge[thick] (12);

\path [->] (123) edge[thick] (23);
\path [->] (123) edge[red,thick] (13);
\path [->] (123) edge[thick] (12);

\path [->] (2345) edge[thick] (345);
\path [->] (2345) edge[red,thick] (245);
\path [->] (2345) edge[red,dashed,thick] (235);
\path [->] (2345) edge[thick] (234);

\path [->] (1345) edge[thick] (345);
\path [->] (1345) edge[thick] (145);
\path [->] (1345) edge[red,thick] (135);
\path [->] (1345) edge[thick] (134);

\path [->] (1245) edge[thick] (245);
\path [->] (1245) edge[thick] (145);
\path [->] (1245) edge[thick] (125);
\path [->] (1245) edge[thick] (124);

\path [->] (1235) edge[thick] (235);
\path [->] (1235) edge[red,dashed,thick] (135);
\path [->] (1235) edge[thick] (125);
\path [->] (1235) edge[thick] (123);

\path [->] (1234) edge[thick] (234);
\path [->] (1234) edge[red,dashed,thick] (134);
\path [->] (1234) edge[red,thick] (124);
\path [->] (1234) edge[thick] (123);

\path [->] (12345) edge[thick] (2345);
\path [->] (12345) edge[red,dashed,thick] (1345);
\path [->] (12345) edge[red,dashed,thick] (1245);
\path [->] (12345) edge[red,thick] (1235);
\path [->] (12345) edge[thick] (1234);

\end{tikzpicture}   
\end{sideways}
\end{center} 
\caption{The graph $G_T$ corresponding to the ideal $I$.}
\label{avramovsexamplegraph}
\end{figure}
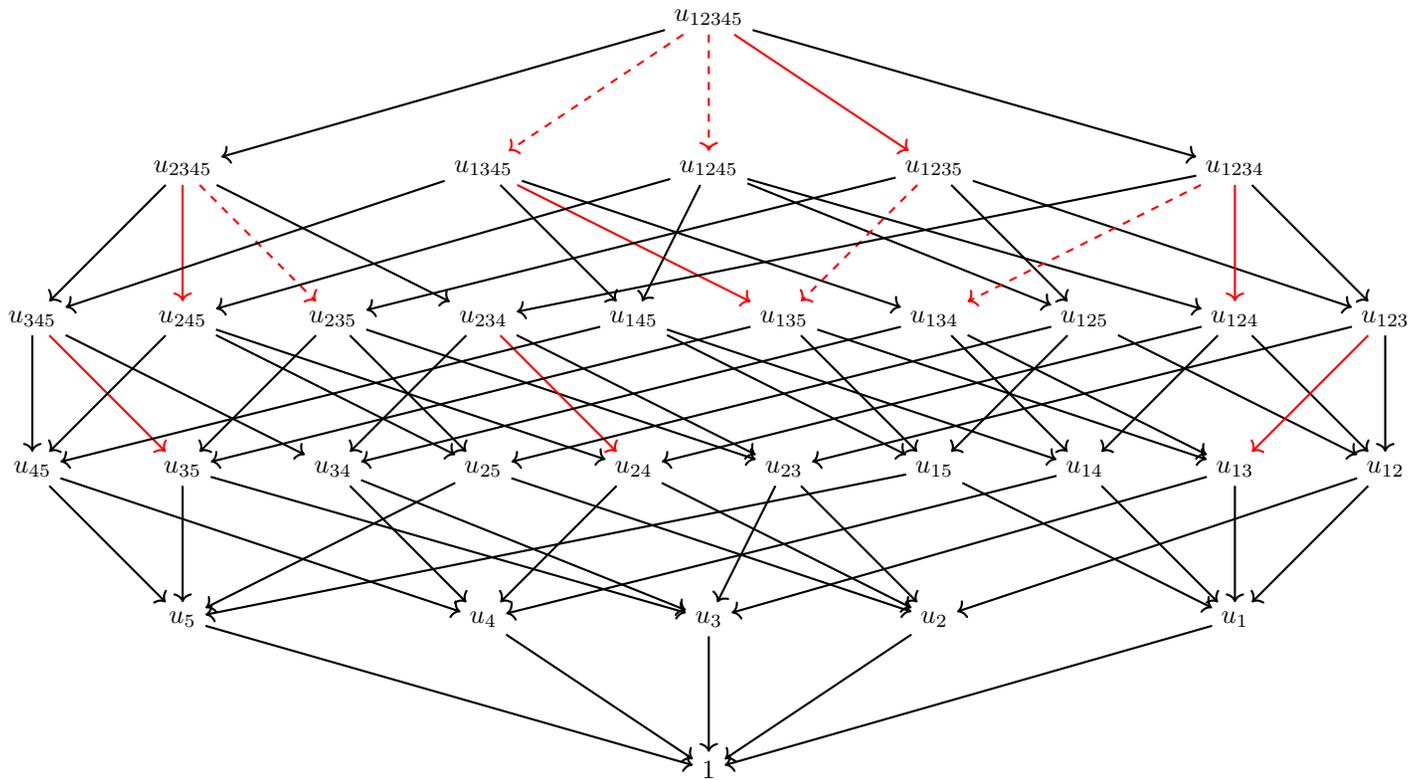

\section*{Acknowledgements}
The author would like to thank his PhD supervisor Jelena Grbi\'c for advice and guidance, 
Fabio Strazzeri and Francisco Belch\'{\i} for useful discussions on respectively algebraic Morse theory and $A_{\infty}$-algebras. Further, the author would like to thank Bernhard K\"ock and Taras Panov for helpful comments on an earlier version of this manuscript.

\clearpage

\bibliographystyle{plain}
\bibliography{main}

\begin{bibdiv}
\begin{biblist}

\bib{avramov1981}{article}{
      author={Avramov, L.},
       title={Obstructions to the existence of multiplicative structures on
  minimal free resolutions},
        date={1981},
        ISSN={0002-9327},
     journal={Amer. J. Math.},
      volume={103},
      number={1},
       pages={1\ndash 31},
         url={http://dx.doi.org/10.2307/2374187},
      review={\MR{601460}},
}

\bib{barneslambe1991}{article}{
      author={Barnes, D.},
      author={Lambe, L.},
       title={A fixed point approach to homological perturbation theory},
        date={1991},
        ISSN={0002-9939},
     journal={Proc. Amer. Math. Soc.},
      volume={112},
      number={3},
       pages={881\ndash 892},
         url={http://dx.doi.org/10.2307/2048713},
      review={\MR{1057939}},
}

\bib{baskakovbuchstaberpanov2004}{article}{
      author={Baskakov, I.},
      author={Bukhshtaber, V.},
      author={Panov, T.},
       title={Algebras of cellular cochains, and torus actions},
        date={2004},
        ISSN={0042-1316},
     journal={Uspekhi Mat. Nauk},
      volume={59},
      number={3(357)},
       pages={159\ndash 160},
         url={https://doi.org/10.1070/RM2004v059n03ABEH000743},
      review={\MR{2117435}},
}

\bib{bayerpeevasturmfels1998}{article}{
      author={Bayer, D.},
      author={Peeva, I.},
      author={Sturmfels, B.},
       title={Monomial resolutions},
        date={1998},
        ISSN={1073-2780},
     journal={Math. Res. Lett.},
      volume={5},
      number={1-2},
       pages={31\ndash 46},
         url={http://dx.doi.org/10.4310/MRL.1998.v5.n1.a3},
      review={\MR{1618363}},
}

\bib{bebengrbic2017}{article}{
      author={Beben, P.},
      author={Grbi\'c, J.},
       title={Configuration spaces and polyhedral products},
        date={2017},
        ISSN={0001-8708},
     journal={Adv. Math.},
      volume={314},
       pages={378\ndash 425},
         url={http://dx.doi.org/10.1016/j.aim.2017.05.001},
      review={\MR{3658721}},
}

\bib{berglund2006}{article}{
      author={Berglund, A.},
       title={Poincar\'e series of monomial rings},
        date={2006},
        ISSN={0021-8693},
     journal={J. Algebra},
      volume={295},
      number={1},
       pages={211\ndash 230},
         url={https://doi.org/10.1016/j.jalgebra.2005.04.008},
      review={\MR{2188858}},
}

\bib{berglundjollenbeck2007}{article}{
      author={Berglund, A.},
      author={J{\"o}llenbeck, M.},
       title={On the {G}olod property of {S}tanley-{R}eisner rings},
        date={2007},
        ISSN={0021-8693},
     journal={J. Algebra},
      volume={315},
      number={1},
       pages={249\ndash 273},
         url={http://dx.doi.org/10.1016/j.jalgebra.2007.04.018},
      review={\MR{2344344}},
}

\bib{burke2015}{article}{
      author={{Burke}, J.},
       title={{Higher homotopies and Golod rings}},
        date={2015-08},
     journal={ArXiv e-prints},
      eprint={1508.03782},
}

\bib{forman1995}{incollection}{
      author={Forman, R.},
       title={A discrete {M}orse theory for cell complexes},
        date={1995},
   booktitle={Geometry, topology, \& physics},
      series={Conf. Proc. Lecture Notes Geom. Topology, IV},
   publisher={Int. Press, Cambridge, MA},
       pages={112\ndash 125},
      review={\MR{1358614}},
}

\bib{forman2002}{article}{
      author={Forman, R.},
       title={A user's guide to discrete {M}orse theory},
        date={2002},
        ISSN={1286-4889},
     journal={S\'em. Lothar. Combin.},
      volume={48},
       pages={Art.\ B48c, 35},
      review={\MR{1939695}},
}

\bib{frankhuizen2016}{article}{
      author={{Frankhuizen}, R.},
       title={{$A_{\infty}$-resolutions and the Golod property for monomial
  rings}},
        date={2016-12},
     journal={ArXiv e-prints},
      eprint={1612.02737},
        note={Forthcoming in Algebr. Geom. Topol.},
}

\bib{golod1978}{article}{
      author={Golod, E.},
       title={Homology of some local rings},
        date={1978},
        ISSN={0042-1316},
     journal={Uspekhi Mat. Nauk},
      volume={33},
      number={5(203)},
       pages={177\ndash 178},
      review={\MR{511890}},
}

\bib{grbictheriault2007}{article}{
      author={Grbi{\'c}, J.},
      author={Theriault, S.},
       title={The homotopy type of the complement of a coordinate subspace
  arrangement},
        date={2007},
        ISSN={0040-9383},
     journal={Topology},
      volume={46},
      number={4},
       pages={357\ndash 396},
         url={http://dx.doi.org/10.1016/j.top.2007.02.006},
      review={\MR{2321037 (2008j:13051)}},
}

\bib{grbictheriault2013}{article}{
      author={Grbi\'c, J.},
      author={Theriault, S.},
       title={The homotopy type of the polyhedral product for shifted
  complexes},
        date={2013},
        ISSN={0001-8708},
     journal={Adv. Math.},
      volume={245},
       pages={690\ndash 715},
         url={http://dx.doi.org/10.1016/j.aim.2013.05.002},
      review={\MR{3084441}},
}

\bib{gulliksenlevin1969}{book}{
      author={Gulliksen, T.},
      author={Levin, G.},
       title={Homology of local rings},
      series={Queen's Paper in Pure and Applied Mathematics, No. 20},
   publisher={Queen's University, Kingston, Ont.},
        date={1969},
      review={\MR{0262227}},
}

\bib{iriyekishimoto2013b}{article}{
      author={Iriye, K.},
      author={Kishimoto, D.},
       title={Decompositions of polyhedral products for shifted complexes},
        date={2013},
        ISSN={0001-8708},
     journal={Adv. Math.},
      volume={245},
       pages={716\ndash 736},
         url={http://dx.doi.org/10.1016/j.aim.2013.05.003},
      review={\MR{3084442}},
}

\bib{jollenbeck2006}{article}{
      author={J{\"o}llenbeck, M.},
       title={On the multigraded {H}ilbert and {P}oincar\'e-{B}etti series and
  the {G}olod property of monomial rings},
        date={2006},
        ISSN={0022-4049},
     journal={J. Pure Appl. Algebra},
      volume={207},
      number={2},
       pages={261\ndash 298},
         url={http://dx.doi.org/10.1016/j.jpaa.2005.10.019},
      review={\MR{2254886}},
}

\bib{jollenbeckwelker2009}{article}{
      author={J\"ollenbeck, M.},
      author={Welker, V.},
       title={Minimal resolutions via algebraic discrete {M}orse theory},
        date={2009},
        ISSN={0065-9266},
     journal={Mem. Amer. Math. Soc.},
      volume={197},
      number={923},
       pages={vi+74},
         url={http://dx.doi.org/10.1090/memo/0923},
      review={\MR{2488864}},
}

\bib{kadeishvili1980}{article}{
      author={Kadei{\v{s}}vili, T.},
       title={On the theory of homology of fiber spaces},
        date={1980},
        ISSN={0042-1316},
     journal={Uspekhi Mat. Nauk},
      volume={35},
      number={3(213)},
       pages={183\ndash 188},
        note={International Topology Conference (Moscow State Univ., Moscow,
  1979)},
      review={\MR{580645}},
}

\bib{katthan2015}{article}{
      author={{Katth{\"a}n}, L.},
       title={{A non-Golod ring with a trivial product on its Koszul
  homology}},
        date={2015-11},
     journal={ArXiv e-prints},
      eprint={1511.04883},
}

\bib{katthan2016}{article}{
      author={Katth{\"a}n, L.},
       title={The {G}olod property for {S}tanley--{R}eisner rings in varying
  characteristic},
        date={2016},
        ISSN={0022-4049},
     journal={J. Pure Appl. Algebra},
      volume={220},
      number={6},
       pages={2265\ndash 2276},
         url={http://dx.doi.org/10.1016/j.jpaa.2015.11.005},
      review={\MR{3448795}},
}

\bib{katthan2016b}{article}{
      author={{Katth{\"a}n}, L.},
       title={{The structure of DGA resolutions of monomial ideals}},
        date={2016-10},
     journal={ArXiv e-prints},
      eprint={1610.06526},
}

\bib{keller2001}{article}{
      author={Keller, B.},
       title={Introduction to {$A$}-infinity algebras and modules},
        date={2001},
        ISSN={1512-0139},
     journal={Homology Homotopy Appl.},
      volume={3},
      number={1},
       pages={1\ndash 35},
      review={\MR{1854636}},
}

\bib{kontsevichsoibelman2009}{incollection}{
      author={Kontsevich, M.},
      author={Soibelman, Y.},
       title={Notes on {$A_\infty$}-algebras, {$A_\infty$}-categories and
  non-commutative geometry},
        date={2009},
   booktitle={Homological mirror symmetry},
      series={Lecture Notes in Phys.},
      volume={757},
   publisher={Springer, Berlin},
       pages={153\ndash 219},
      review={\MR{2596638}},
}

\bib{lupalmieriwuzhang2009}{article}{
      author={Lu, D.-M.},
      author={Palmieri, J.},
      author={Wu, Q.-S.},
      author={Zhang, J.},
       title={{$A$}-infinity structure on {E}xt-algebras},
        date={2009},
        ISSN={0022-4049},
     journal={J. Pure Appl. Algebra},
      volume={213},
      number={11},
       pages={2017\ndash 2037},
         url={http://dx.doi.org/10.1016/j.jpaa.2009.02.006},
      review={\MR{2533303}},
}

\bib{massey1958}{incollection}{
      author={Massey, W.},
       title={Some higher order cohomology operations},
        date={1958},
   booktitle={Symposium internacional de topolog\'\i a algebraica
  {I}nternational symposium on algebraic topology},
   publisher={Universidad Nacional Aut\'onoma de M\'exico and UNESCO, Mexico
  City},
       pages={145\ndash 154},
      review={\MR{0098366}},
}

\bib{merkulov1999}{article}{
      author={Merkulov, S.},
       title={Strong homotopy algebras of a {K}\"ahler manifold},
        date={1999},
        ISSN={1073-7928},
     journal={Internat. Math. Res. Notices},
      number={3},
       pages={153\ndash 164},
         url={http://dx.doi.org/10.1155/S1073792899000070},
      review={\MR{1672242}},
}

\bib{mermin2012}{incollection}{
      author={Mermin, J.},
       title={Three simplicial resolutions},
        date={2012},
   booktitle={Progress in commutative algebra 1},
   publisher={de Gruyter, Berlin},
       pages={127\ndash 141},
      review={\MR{2932583}},
}

\bib{millersturmfelsyanagawa2000}{article}{
      author={Miller, E.},
      author={Sturmfels, B.},
      author={Yanagawa, K.},
       title={Generic and cogeneric monomial ideals},
        date={2000},
        ISSN={0747-7171},
     journal={J. Symbolic Comput.},
      volume={29},
      number={4-5},
       pages={691\ndash 708},
         url={http://dx.doi.org/10.1006/jsco.1999.0290},
        note={Symbolic computation in algebra, analysis, and geometry
  (Berkeley, CA, 1998)},
      review={\MR{1769661}},
}

\bib{serre1965}{book}{
      author={Serre, J.-P.},
       title={Alg\`ebre locale. {M}ultiplicit\'es},
      series={Cours au Coll\`ege de France, 1957--1958, r\'edig\'e par Pierre
  Gabriel. Seconde \'edition, 1965. Lecture Notes in Mathematics},
   publisher={Springer-Verlag, Berlin-New York},
        date={1965},
      volume={11},
      review={\MR{0201468}},
}

\bib{skoldberg2006}{article}{
      author={Sk\"oldberg, E.},
       title={Morse theory from an algebraic viewpoint},
        date={2006},
        ISSN={0002-9947},
     journal={Trans. Amer. Math. Soc.},
      volume={358},
      number={1},
       pages={115\ndash 129},
         url={http://dx.doi.org/10.1090/S0002-9947-05-04079-1},
      review={\MR{2171225}},
}

\bib{stasheff1963}{article}{
      author={Stasheff, J.},
       title={Homotopy associativity of {$H$}-spaces. {I}, {II}},
        date={1963},
        ISSN={0002-9947},
     journal={Trans. Amer. Math. Soc. 108 (1963), 275-292; ibid.},
      volume={108},
       pages={293\ndash 312},
      review={\MR{0158400}},
}

\bib{taylor1966}{book}{
      author={Taylor, D.},
       title={Ideals generated by monomials in an {$R$}-sequence},
   publisher={Ph. D. Thesis, University of Chicago},
        date={1966},
      review={\MR{2611561}},
}

\end{biblist}
\end{bibdiv}
\end{document}